\documentclass[11pt]{article}
\usepackage{amsmath,amssymb,amsthm}
\usepackage{mathrsfs}
\usepackage[margin=1in]{geometry}
\usepackage{hyperref}
\hypersetup{colorlinks=true,linkcolor=blue,citecolor=blue,urlcolor=blue,
      pdftitle={Amplified moments and a refinement of the quasi-Riemann hypothesis},
  pdfauthor={Yuanda Xu, Hejian Sang}}
\newtheorem{theorem}{Theorem}[section]
\newtheorem{proposition}[theorem]{Proposition}
\newtheorem{lemma}[theorem]{Lemma}
\newtheorem{corollary}[theorem]{Corollary}
\theoremstyle{remark}
\newtheorem{remark}[theorem]{Remark}
\newcommand{\N}{\mathrm N}
\newcommand{\OO}{\mathcal O}
\newcommand{\eps}{\varepsilon}
\newcommand{\Rea}{\operatorname{Re}}
\newcommand{\Bnew}{B_1}
\title{Amplified moments and a refinement\\
of the quasi-Riemann hypothesis}
\author{Yuanda Xu \quad Hejian Sang}
\date{October 2026}

\begin{document}
\maketitle

\begin{abstract}
Recently, OpenAI \cite{OAI26} announced a uniform \(7/8\)
zero-free half-plane for finite-order Hecke \(L\)-functions
over \(\mathbb Q(\sqrt{-3})\), together with its transfer
to Dirichlet \(L\)-functions. Assuming the analytic inputs
of that work, we improve the boundary from \(0.875\)
to \(0.8748867\) for the same families.
The improvement combines a relaxed direct estimate with
an amplified first moment of the reflected numerator,
obtained from the existing fourth-moment estimate.
A two-stage parameter argument yields the stated boundary
without strengthening the inverse-moment estimate.
\end{abstract}

\begin{center}
\begin{minipage}{0.9\textwidth}

\end{minipage}
\end{center}

\clearpage
\tableofcontents
\clearpage

\section{Introduction}\label{sec:introduction}

\subsection{Background}

The Riemann hypothesis asserts that \(\zeta(s)\ne0\)
for \(\Rea s>1/2\). The quasi-Riemann hypothesis asks
only for some fixed \(\sigma<1\) such that \(\zeta(s)\)
has no zeros in \(\Rea s>\sigma\).
Classical zero-free regions of de la Vall\'ee Poussin
and Vinogradov--Korobov approach the line \(\Rea s=1\)
as the height grows; see \cite{Titchmarsh}.
Zero-density estimates, including the recent work of
Guth and Maynard \cite{GuthMaynard}, bound the number
of zeros to the right of a given line, but do not by
themselves exclude every zero in a fixed half-plane.

Moment estimates provide a different route to zero-free
statements. Bettin and Gonek \cite{BettinGonek} show
that the \(\theta=\infty\) conjecture for mollified
second moments implies the Riemann hypothesis.
The method considered here also uses moments, but of
residue-character polynomials over an auxiliary family;
it does not assume that mollified-moment conjecture.

The preprint \cite{OAI26} develops a method based on
completed cubic-theta sums, finite compensation, and
moments of residue-character polynomials. It claims
the zero-free half-plane \(\Rea s>7/8\) for all
finite-order Hecke \(L\)-functions over
\(F=\mathbb Q(\sqrt{-3})\), and hence for all Dirichlet
\(L\)-functions.

The cubic-theta background originates in Patterson's
construction \cite{Patterson} and the work of
Heath-Brown and Patterson on Kummer sums
\cite{HBPatterson}. Dunn and Radziwi\l\l\
\cite{DunnRadziwill} develop level-aspect summation
formulas for cubic Gauss sums. The related large-sieve
theory includes Heath-Brown's cubic estimate
\cite{HBcubic}, the quadratic number-field inequality
of Goldmakher and Louvel \cite{GoldmakherLouvel}, and
the higher-order character estimates of Blomer,
Goldmakher and Louvel \cite{BGL}.
These works explain the analytic background; the
specific reflection and moment estimates used below
are imported from \cite{OAI26} with the conventions
of Section~\ref{sec:inputs}.

Liu \cite{Liu26} explains this method and improves its
parameter comparison. In particular, the stronger
algebraic boundary obtained there is
\begin{equation}\label{eq:liu-boundary}
 B_{\mathrm L}=\frac{1507-2\sqrt{921}}{1653}
              =0.874957069799\ldots .
\end{equation}
We change the high-side estimate itself: instead of
using only a pointwise bound for the reflected numerator,
we average it over the detected row class with the aid
of the same physical prime polynomials that enter the
compensated probe.

In a companion manuscript \cite{XuSextic}, we prove a matching
lower bound of scale \(A^{5/6}B^{1/3}\), up to arbitrarily
small powers, for the squared norm of the sextic residue-character
matrix with squarefree indices over the Eisenstein integers,
in the range \(B^{16/11+\delta}\le A\le B^{2-\delta}\),
\(0<\delta<3/11\). That argument evaluates an exceptional
sixth-power contribution to a Gauss-sum second moment and
shows that it survives the squarefree restrictions. It is
closely related to the Gauss--Gram analysis underlying the
direct estimate here, where a second-moment upper bound is
combined with reflected energy for the compensated probe.
The two results are complementary but logically separate:
the large-sieve lower bound does not use the \(7/8\) zero-free
conclusion, and the present refinement does not use that
lower bound. In particular, a lower bound for the operator
norm does not establish optimality of the zero-free boundary
obtained from this particular probe.

\subsection{Results}

\begin{theorem}\label{thm:main}
Assume the analytic inputs of \cite{OAI26} recorded in
Section~\ref{sec:inputs}, including the family-wide
\(7/8\) zero-free half-plane.
Every finite-order Hecke \(L\)-function over
\(F=\mathbb Q(\sqrt{-3})\) has no zeros in
\[
 \Rea s>\Bnew=\frac{8748867}{10000000}
             =\frac78-\frac{1133}{10000000}.
\]
The same holds for every Dirichlet \(L\)-function,
including \(\zeta(s)\). Poles at \(s=1\) for principal
characters are allowed.
\end{theorem}

The first stage uses the \(7/8\) conclusion of
\cite{OAI26} to license its moment parameter.
The additional numerator estimate and the parameter
and contour extensions are proved below.

A smaller boundary gives a larger zero-free region.
For comparison, the proof passes through
\[
 \begin{array}{c|c|l}
 \text{boundary}&\text{decimal value}&\text{source}\\ \hline
 7/8&0.875&\cite{OAI26}\\
 B_{\mathrm L}&0.874957069799\ldots&\cite{Liu26}\\
 B_0=69991/80000&0.8748875&
       \text{Lemma~\ref{lem:closure} with }j=0\\
 \Bnew=8748867/10000000&0.8748867&
        \text{Theorem~\ref{thm:main}}.
 \end{array}
\]
The improvement over \eqref{eq:liu-boundary} is
\[
 B_{\mathrm L}-\Bnew=0.0000703697991687\ldots>0.
\]
For a rational verification of the strict comparison,
one may use \(\sqrt{921}<30348/1000\).

\subsection{Proof roadmap: one probe, two estimates}

The method of \cite{OAI26} compares two descriptions of
one arithmetic expression. Its direct, or low, estimate
comes from reflection and an additive Gram bound. Its
Poisson expansion contains a principal row whose
residue is an inverse Mellin transform of the
reciprocal of the target \(L\)-function. The remaining
rows must have a uniform power saving.

Here the principal row refers to the auxiliary Poisson
index \(u=1\), not to a restriction that the target
character be principal. Its residue contains the
reciprocal of the arbitrary target \(L\)-function.
The final continuation argument requires a saving
common to all targets, although constants and lower
thresholds may depend on each fixed target.

To make the comparison explicit, write \(\beta_*\) for the
supremum of real parts of zeros in the target family, defined
precisely in \eqref{eq:beta}. The total prime-length exponent
\(\ell\) determines the candidate boundary
\(B(\ell)=11/12-\ell/4\). A second geometric parameter
\(b>0\) determines the affine exponent
\(C_b(s)=s-(4+b)/6\), which translates a location in the
\(s\)-plane into a growth rate for the probe.
After scalar normalization, the arithmetic expression is
\(J_\eta(Z)\). Its principal Mellin term \(f_\eta(Z)\)
contains \(H_\eta(s)/L_F^S(s,\eta)\), with \(H_\eta\)
holomorphic and nonvanishing in the region of interest.
The two estimates to be proved are
\[
 \begin{aligned}
 |J_\eta(Z)|&\ll_{\eta,\eps}
                  Z^{C_b(B(\ell))+\eps}
                  &&\text{(direct estimate)},\\
 |J_\eta(Z)-f_\eta(Z)|&\ll_\eta
                  Z^{C_b(\beta_*)-\rho}
                  &&\text{(Poisson remainder)},
 \end{aligned}
\]
where the second holds under the contradiction hypothesis
\(\beta_*>B(\ell)\), with \(\rho>0\) common to all targets.
Since \(C_b\) has slope one, these bounds force
\(f_\eta\) to grow strictly more slowly than
\(Z^{C_b(\beta_*)}\). Mellin transformation then continues
\(1/L_F^S(s,\eta)\) to a larger half-plane. The factor
\(H_\eta\) cannot cancel a pole of this reciprocal.
The common saving contradicts the definition of \(\beta_*\).
This final implication is isolated in Lemma~\ref{lem:closure}.

The main work is therefore to bound the Poisson remainder.
For each class of nonprincipal rows we need both its size
and the average of its reflected numerator. The detector
and the inverse/plain moments give the row count; the new
amplified numerator estimate uses the prime amplitudes on
that same class. Section~\ref{sec:high} combines these two
outputs with the contour and scale factors. Its endpoint
inequality is then a numerical condition, certified separately
in Section~\ref{sec:parameters}.

\paragraph{The three refinements.}

There are three refinements in the present argument.
First, the low estimate retains the decay attached to
a rescaled prime tuple. This permits the total prime
length \(\ell\) to exceed \(1/6\), without a stronger
reflected-energy theorem
(Proposition~\ref{prop:low}).

Second, a fourth-moment estimate for plain
polynomials is applied to an approximate functional
equation for the numerator. Short polynomial blocks
leave room for more prime amplification in the moment;
long blocks instead benefit from stronger coefficient
decay. We choose the numerator contour so that these
two effects give the same saving for every block.
If the shortest blocks require more primes than are
available, their lengths are small enough for a direct bound
(Proposition~\ref{prop:numerator}).

Third, the moment parameter \(\kappa\) measures the cost
of prime amplification in the plain fourth moment.
A smaller value permits more amplification, but requires
the family bound \(\beta_*\le(1+\kappa)/2\).
The first improved half-plane permits a
smaller fixed value of \(\kappa\), using
Lemma~\ref{lem:kappa-extension} and the licensing
identity \eqref{eq:license}. The possibility
of such parameter feedback is already discussed in
\cite{Liu26}; the principal new analytic step here
is the averaged numerator estimate.

The logical order is
\begin{equation}\label{eq:bootstrap-order}
 \beta_*\le\frac78
 \ \Longrightarrow\ \kappa=\frac34
 \ \Longrightarrow\ \beta_*\le B_0
 \ \Longrightarrow\ \kappa=\frac{29991}{40000}
 \ \Longrightarrow\ \beta_*\le\Bnew.
\end{equation}
The second value of \(\kappa\) is used only after
the bound that licenses it has been proved.

\subsection{Organization}

The proof has three interfaces. Proposition~\ref{prop:low}
constructs the normalized probe and its principal term and
supplies the direct bound. Propositions~\ref{prop:row-count}
and \ref{prop:numerator} supply the two row estimates used
in the remainder bound of Proposition~\ref{prop:high}.
Finally, Proposition~\ref{prop:parameters} certifies the
numerical hypotheses, and Lemma~\ref{lem:closure} turns
each licensed parameter choice into a zero-free boundary.
Section~\ref{sec:inputs} records the imported hypotheses
needed by these statements.
The appendices contain the parameter-range extension,
the approximate-functional-equation uniformity,
the Euler and contour calculations, the exact
certificates, and the common order of choices.
Appendix~\ref{app:sources} collects the precise source
references. Readers familiar with the source probe
and its two moments may proceed to
Section~\ref{sec:numerator} for the main new estimate.

\section{Notation and analytic input}\label{sec:inputs}

Throughout the paper, results imported from \cite{OAI26}
are hypotheses of Theorem~\ref{thm:main} and its consequences.
We retain the source character presentations, coefficient
classes, and uniformity conventions. The extensions of
the moment range and contour regions required here are
proved in Appendices~\ref{app:kappa} and \ref{app:contours}.

The auxiliary family serves two purposes. Its zeros
determine where the denominator contour can be moved,
and its polynomial moments limit the number of rows
on which that contour must remain far to the right.
We first define this family, then state the two moments
and the detector that connects them to zeros.
The full source index is in Appendix~\ref{app:sources}.

\subsection{Characters, zeros, and physical rows}

Write \(\OO=\mathbb Z[\omega]\), where
\(\omega^2+\omega+1=0\), and let \(\N\) denote the ideal
norm, also used for the norm of a nonzero element.
Ideal sums use the primary-generator convention of
\cite{OAI26}. Put
\[
 \chi_n(u)=\left(\frac{u}{n}\right)_6,
\]
with value zero when \(u\) and \(n\) are not coprime.
Every character presentation retains its original
zero extension, even if its primitive inducing
character has a smaller defining modulus.

Let
\begin{equation}\label{eq:beta}
 \beta_*=\sup\left(\{1/2\}\cup
 \left\{\Rea\rho:
 \begin{array}{l}
 L_F(\rho,\eta)=0,\quad1/2\le\Rea\rho\le1,\\
 \eta\text{ primitive and of finite order}
 \end{array}\right\}\right).
\end{equation}
Poles are not counted. We use the result
\(\beta_*\le7/8\) of \cite{OAI26} to start
\eqref{eq:bootstrap-order}.
This is a family-wide parameter, not the zero boundary of
one fixed target. It both licenses the moment estimate and
provides the supremum contradicted in Lemma~\ref{lem:closure}.

For a finite set \(S\) of prime ideals,
\[
 L_F^S(s,\eta)=L_F(s,\eta)
              \prod_{p\in S}(1-\eta(p)(\N p)^{-s}).
\]
These extra factors are nonzero for \(\Rea s>0\).
The fixed set \(S\), finite ray data \(T\), and
normalization data are those of the source probe.
They may depend on a fixed target \(\eta\), but
not on the large parameter \(Z\).
The physical rows are nonzero sixth-power-free
elements \(u\) with \((u,S)=1\), including their
unit factors.
They are the indices that occur in the Poisson
expansion of the probe. The adjective ``physical''
distinguishes these indices from the auxiliary
character presentations used to estimate their terms.

Put \(\Theta=\langle\eta,\widehat T\rangle\); this finite
group collects the fixed twists introduced by the target
and the ray restrictions.
The detector uses the finite family of presentations
\[
 \mathcal X_u=
 \{\nu(n)\chi_n(u)^{\epsilon_\chi}:
       \nu\in\Theta,\ \epsilon_\chi\in\{-1,1\}\}.
\]
The inverse orientation also has value zero on
nonunits; it is not a reciprocal evaluated at zero.
If \(p\notin S\) divides a physical row with
\(v_p(u)=j\in\{1,\ldots,5\}\), its local character
has order \(6/\gcd(6,j)>1\).
No member of \(\Theta\), all of which are unramified
at \(p\), can cancel it.
Thus nonunit physical rows induce outside \(\Theta\).
This verifies the nonexceptional-character condition needed
for positive prime slots in the plain moment. The bounded
unit rows are treated separately.

\subsection{Polynomial conventions and the two moments}

The detector will produce two simultaneous witnesses on a
difficult row: an inverse polynomial and a plain polynomial.
Their moment bounds give complementary estimates for the
number of such rows in Section~\ref{sec:rows}. The plain
moment has a second use in Section~\ref{sec:numerator}:
it estimates the blocks of the physical numerator's
approximate functional equation. That numerator need not
have the detector's fixed twist; the information shared
between the two applications is the row class and its
physical prime amplitudes, not an identification of characters.

At base \(U\), define
\[
 \begin{split}
 M_\psi(r;W)&=U^{-r/2}
       \sum_n\mu(n)\psi(n)W(\N n/U^r),\\
 S_\psi(m;W)&=U^{-m/2}
       \sum_n\psi(n)W(\N n/U^m).
 \end{split}
\]
The normalization is by the square root of the
summation length. The inverse polynomial \(M_\psi\)
has M\"obius coefficients, as in the Dirichlet series
for \(1/L(s,\psi)\). The plain polynomial \(S_\psi\)
has no M\"obius factor and is the type of block that
appears in an approximate functional equation for
\(L(s,\psi)\). Products of short prime sums will
amplify the information supplied by these polynomials.

The weights are smooth annular functions. A pure norm
twist is recorded as a separate height parameter.
All common masks and fixed characters in a product
are retained. In particular, two plain factors in
one source moment have the same row character,
including its fixed finite-ray twist.

We use the following results of \cite{OAI26}.
Statement numbers refer to the version dated
September 30, 2026.
The marked inverse moment \cite[Lemma~17.1]{OAI26} gives
\begin{equation}\label{eq:marked-moment}
 \sum_{0<\N u\ll U}|M_{\psi_u}(r;W)Q_u|^2
       \ll U^{1+\eps}
\end{equation}
when the selected prime factors have total length \(z\)
and
\[
 r+2z<1,\qquad 2r+8z<3.
\]
Fixed positive margins are used when this estimate
is invoked. Its slot coefficients may be bounded
independently for each prime and slot, but are
independent of the row.
The sixth-power amplification \cite[Lemma~17.6]{OAI26} gives
\begin{equation}\label{eq:long-inverse}
 \sum_{\N u\asymp U}^{(6)}
       |M_{\psi_u}(r;W)|^2
 \ll U^{e(r)+\eps},\qquad
 e(r)=\max\{1,(1+5r)/6\}.
\end{equation}
The superscript \((6)\) denotes the physical
sixth-power-free rows.

The plain fourth moment \cite[Lemma~18.1]{OAI26}
gives
\begin{equation}\label{eq:plain-moment}
 \sum_u
 |S_{\psi_u}(m_1;W_1)S_{\psi_u}(m_2;W_2)Q_u|^2
       \ll U^{1+\eps}.
\end{equation}
With no prime slots, there is no restriction on the
bounded nonnegative lengths \(m_1,m_2\), and only
principal inducing characters are excluded.
With positive total slot length \(z\), the conditions are
\begin{equation}\label{eq:plain-capacity}
 m_1+m_2+6\kappa z\le1,\qquad
 \beta_*\le(1+\kappa)/2\quad(\kappa<1).
\end{equation}
The inducing row character must then lie outside
\(\Theta\). Positive-slot coefficients are fixed finite
linear combinations of characters in \(\Theta\);
arbitrary arithmetic coefficients are not allowed.
The underlying prime supports are disjoint before
the common masks and row zeros are imposed.
The source statement has \(3/4\le\kappa\le1\).
The first condition in \eqref{eq:plain-capacity} is a length
budget: the two polynomial blocks and the selected primes
must fit in one moment. The second condition licenses that
budget from a known family zero bound. Thus reducing
\(\kappa\) increases prime capacity, but can only be done
after proving the stronger bound that licenses it.

\begin{lemma}\label{lem:kappa-extension}
The plain moment \eqref{eq:plain-moment}, with all
the preceding coefficient and uniformity conditions,
continues to hold for
\[
 13/18\le\kappa\le1.
\]
Its positive-slot hypothesis
\(\beta_*\le(1+\kappa)/2\) is unchanged.
\end{lemma}

The proof is given in Appendix~\ref{app:kappa}.
It modifies the numerical bounds in the source
finite induction; it does not invoke the original
theorem outside its stated range.

\paragraph{Uniformity.}
In each moment, real lengths range over prescribed
bounded sets. The plain-moment slot mesh depends
on these sets and the requested power loss, not
on the eventual fixed slot count.
For that fixed count and fixed arithmetic data,
the estimates use finitely many smooth seminorms
and a fixed polynomial in the separated heights.
The source Sobolev argument permits rowwise smooth
parameters. It does not enlarge the arithmetic
coefficient class.
All applications here have effective row width one
and no additional moving fixed twist.
The amplitude \(q\) introduced below is not a
moving-support parameter.

\subsection{Buffered bins and simultaneous witnesses}

The detector groups rows according to how far their
denominator contours can be shifted. Each group has
a boundary \(a\) and a height interval; the extra
margin in the real part and height is the buffer
needed for subsequent contour moves. A row assigned
above the lowest group must have large polynomial
witnesses. The moment bounds will make such rows rare.

Fix \(U=Z^d\) in a bounded positive range of \(d\),
and a height \(T_1=Z^\tau\), with \(\tau\) chosen
late in the argument. The source buffered-bin
lemma associates to each nonunit physical row
an index \(i\) and a grid point \(a\ge51/100\).
Put
\[
 \delta=2a-1.
\]
Here \(a\) records how far left the denominator contour is
permitted to move on this class; \(\delta\) expresses the
same location in the normalization used by the witness
exponents. Neither is a new global candidate zero boundary.
There are \(O_e(1)\) pairs \((i,a)\), where \(e>0\)
is the fixed bin width. Every presentation in
\(\mathcal X_u\) satisfies
\begin{equation}\label{eq:buffered}
 |L^S(s,\psi)|+|L^S(s,\psi)^{-1}|
       \ll U^\eps
\end{equation}
for \(\Rea s\ge a+6e\) and the assigned buffered
heights. The parameter is exact in the sense
\[
 a\le\beta_*,\qquad \delta\le2\beta_*-1.
\]
At the floor \(a=51/100\), use the trivial row count.

Above the floor, \cite[Proposition~8.3]{OAI26}
gives, for every \(t\in[1,3/2]\), simultaneous
witnesses \(M_r,S_m\), with one common presentation
in \(\mathcal X_u\), such that, up to arbitrarily
small losses,
\begin{equation}\label{eq:witnesses}
 \begin{gathered}
 r\le t,\qquad r+m\ge t,\qquad 0\le m\le1/2,\\
 |M_r|^2\gg U^{\delta r-\eps},\qquad
 |S_m|^2\gg U^{\delta m-\eps}.
 \end{gathered}
\end{equation}
The precise dyadic endpoint errors are
\(O(1/\log U)\) or \(O(\eps)\).
Subdividing by presentation and dyadic lengths
costs only a logarithmic factor; the remaining
rowwise smooth parameters have the stated
polynomial height cost.
We keep the source's single cumulative allowance
for all added frequencies.

\section{The compensated probe and its low estimate}
\label{sec:probe}

We construct the expression \(J_\eta\) in the roadmap.
Its normalization must make the principal Mellin term
contain exactly \(H_\eta/L_F^S\), with \(H_\eta\)
nonvanishing; its other rows must contain the prime factors
used for amplification. Both features come from one finite
compensation operation on the source probe. We first define
that arithmetic operation as \(\mathscr I_\eta\), then
divide out its scalar principal factor to obtain \(J_\eta\).
The direct and Poisson estimates will use this same normalization.

\subsection{One physical expression}

Choose \(b>0\) and \(1/6\le\ell\le1/5\), and put
\begin{equation}\label{eq:geometry}
 l_x=\frac{1-b-\ell}{2},\qquad
 l_y=\frac{1+b-\ell}{2},\qquad
 h=\frac{1+b+3\ell}{2}.
\end{equation}
Then \(h=1-l_x+\ell=l_y+2\ell\).
The independent choices are \(b\) and \(\ell\): \(b\)
sets the imbalance \(l_y-l_x\), and \(\ell\) is the total
prime-length exponent. The derived exponents \(l_x,l_y\)
set the two probe scales, while \(h\) is the critical
Poisson-row scale, at which the high-side exponent will be tested.
We will require
\begin{equation}\label{eq:low-geometry}
 l_x>\ell,\qquad l_y-\ell>11b/6.
\end{equation}
Define
\[
 X=Z^{l_x},\qquad Y=Z^{l_y},\qquad
 C_b(s)=s-\frac{4+b}{6},\qquad
 B(\ell)=\frac{11}{12}-\frac{\ell}{4}.
\]
The map \(C_b\) converts a Mellin location into a power of
\(Z\). The candidate boundary \(B(\ell)\) is the location
whose power equals the direct estimate:
\begin{equation}\label{eq:signal-low}
 C_b(B(\ell))=\frac{l_x}{2}+\frac b{12}.
\end{equation}

Choose an even number \(K\) of physical prime
windows with lengths \(\ell_i>0\), total length
\(\sum_i\ell_i=\ell\), and scales \(P_i=Z^{\ell_i}\).
Each window supplies one prime in the tuple; we call
that position a prime slot. Thus a slot length is
the exponent of its prime scale, not a prime count.
Their weights \(W_i\) are nonnegative, nonzero,
smooth and annular. The underlying windows are
disjoint before imposing \(p\notin S\) and \(p\in1_T\).
The count \(K\), lengths and weights are fixed
before \(Z\).
The individual windows make whole-slot selection possible
in the two moment arguments. Their total length \(\ell\)
controls the boundary, whereas their sufficiently fine mesh
controls the loss when a fractional selection is rounded.

Compensation alternates between imposing divisibility and
rescaling the source probe. The subset index below specifies
which operation is taken at each prime; it is finite
bookkeeping, not another asymptotic parameter.
Let \(I_{\eta;D}(X,Y,Z)\) be the source completed
probe with \(1_{D\mid cn^3}\) imposed on its
completed index; \(I_{\eta;1}=I_\eta\).
For \(J\subseteq\{1,\ldots,K\}\), put
\(p_J=\prod_{i\in J}p_i\), with \(p_\varnothing=1\).
The physical compensated expression is
\begin{equation}\label{eq:physical-probe}
 \begin{split}
 \mathscr I_\eta(Z)
 ={}&\sum_{(p_i)\in\prod_i\mathcal P_i(Z)}
          \prod_iW_i(\N p_i/P_i)
       \sum_{J\subseteq\{1,\ldots,K\}}
       (-1)^{|J|}(\N p_J)^{-3/2}
                    \overline{\eta(p_{J^c})}\\
 &\hspace{16mm}\cdot
 I_{\eta;p_{J^c}}
    (X/\N p_J,Y/\N p_J,Z\N p_{J^c}).
 \end{split}
\end{equation}
This is the exact finite operation of \cite{OAI26}.
Its prime windows stay at their original scales
in every summand.
In particular, \(\mathscr I_\eta\) is defined
arithmetically, not by a high-side Mellin integral.

\subsection{Retaining the external decay}

Longer prime windows improve the candidate boundary
\(B(\ell)=11/12-\ell/4\), but they also enlarge a
completed-row norm in the direct estimate. The point
is to keep the decay from each rescaled term until
after that norm is estimated. This decay absorbs
the excess throughout \(\ell\le1/5\), as the final
inequality \eqref{eq:low-saving} will show.

The proposition records both the direct estimate and the
principal term needed for the later comparison. Its scalar
\(A_T(Z)\) removes the leading prime-window mass, at only
a logarithmic cost. The normalized probe \(J_\eta\) is
then the object estimated on both sides. The Euler correction
\(H_\eta\) remains inside the principal integral; its
nonvanishing ensures that a zero of the target \(L\)-function
cannot be canceled there. Finally, \(f_\eta\) denotes that
principal integral, whose growth will determine continuation
of the reciprocal in Lemma~\ref{lem:closure}.

\begin{proposition}[Normalized probe and low bound]\label{prop:low}
Assume the analytic inputs of Section~\ref{sec:inputs}.
Let \(b>0\), \(1/6\le\ell\le1/5\), and let
\(l_x=(1-b-\ell)/2\), \(l_y=(1+b-\ell)/2\),
\(h=(1+b+3\ell)/2\) satisfy
\(l_x>\ell\) and \(l_y-\ell>11b/6\).
For any fixed even system of disjoint prime windows
specified above, the arithmetic probe
\eqref{eq:physical-probe} satisfies
\[
 |\mathscr I_\eta(Z)|
       \ll_{\eta,\eps} Z^{l_x/2+b/12+\eps}.
\]
Define
\begin{equation}\label{eq:normalization}
 \begin{split}
 A_T(Z)&=(-1)^K Z^{-\ell/6}
      \prod_i\sum_{p\in\mathcal P_i(Z)}
                  W_i(\N p/P_i)(\N p)^{-5/6},\\
 J_\eta(Z)&=\mathscr I_\eta(Z)/(c_S A_T(Z)),
 \end{split}
\end{equation}
where \(c_S>0\) is the product of the source scalar
residues and fixed Mellin factors. Then
\(A_T(Z)\asymp(\log Z)^{-K}>0\) for all sufficiently
large real \(Z\), and
\begin{equation}\label{eq:normalized-low}
 |J_\eta(Z)|\ll_{\eta,\eps}Z^{C_b(B(\ell))+\eps},
 \qquad C_b(s)=s-(4+b)/6,\quad B(\ell)=11/12-\ell/4.
\end{equation}
The same probe has the absolutely convergent Poisson
representation \eqref{eq:high-identity}. Its principal
row determines a function \(H_\eta\), holomorphic on
\(\Rea s>87/100\), with \(|H_\eta(s)-1|\le1/2\)
after a fixed initial-prime cutoff. The associated
principal Mellin term is
\begin{equation}\label{eq:principal-signal}
 f_\eta(Z)=\frac1{2\pi i}\int_{(2)}
 Z^{C_b(s)}e^{(s-5/6)^2}
              \frac{H_\eta(s)}{L_F^S(s,\eta)}\,ds.
\end{equation}
All arithmetic data and windows are fixed before
\(Z\); the implied constants may depend on those
data and the fixed target \(\eta\).
\end{proposition}

\begin{proof}
\emph{Direct estimate.}
For one rescaled subset \(J\), write
\[
 d_J=\sum_{i\in J}\ell_i=d,\quad
 X'\asymp XZ^{-d},\quad Y'\asymp YZ^{-d},\quad
 M'=1-\ell-2d,\quad \ell'=\ell-d.
\]
The source completed-row argument applies
\cite[Lemma~14.3]{OAI26} at the actual row and dual
scales. Before its special substitution
\(\ell=1/6\), its two upper branches are
\[
 M'-O/2-\Delta_H+o(1),\qquad
 M'+\frac{1+3\ell'-2M'-2\Delta_H}{4}+o(1).
\]
Here \(O\ge-o(1)\) is the powerful-row length,
and \(\Delta_H\ge-o(1)\) is the actual row deficit.
The fixed row parts have already been counted.
The structural relation used in the dual-length
bound is still
\[
 M'+\ell'-1=-3d.
\]
Thus the same argument gives the completed-row
norm
\[
 \sum_m|B_m^J|^2
 \ll Z^{M'+L_\ell(d)+\eps},\qquad
 L_\ell(d)=\frac{(5\ell+d-1)_+}{4}.
\]
The Gaussian annular decomposition and common
Fourier separation are unchanged. Its exact
annular offsets enter the prescribed small loss.

The quantitative additive Gram bound,
\cite[Proposition 15.2]{OAI26}, is already stated
for general polynomially bounded \(Q,Y'\) with
\(P_a=Y'^2/Q\ge1\). Here \(Q\asymp X'Y'\) and
\(P_a\asymp Z^b\). It gives
\[
 \sum_m|A_m|^2\ll
 \frac Q{Y'}
 \left(1+P_a^{1/6}+\frac{P_a^2}{Y'}\right)
 Z^\eps
\]
apart from its integrable polynomial height factor.
The second inequality in \eqref{eq:low-geometry}
absorbs \(P_a^2/Y'\) into \(P_a^{1/6}\),
uniformly for \(0\le d\le\ell\).
With the exact low-separation prefactor
\(Q^{-1/2}\), Cauchy--Schwarz yields the
per-tuple bound
\[
 (X')^{1/2}P_a^{1/12}
                     Z^{L_\ell(d)/2+\eps}.
\]
Relative to the unrescaled estimate, the tuple
count, compensation coefficient and \(X'\)
factor contribute \(d-3d/2-d/2=-d\).
Since \(\ell\le1/5\),
\begin{equation}\label{eq:low-saving}
 -d+\frac{(5\ell+d-1)_+}{8}\le-\frac{7d}{8}\le0.
\end{equation}
The marked slots are already inside \(B_m^J\);
they are not counted again.
All scales are positive powers of \(Z\) by
\eqref{eq:low-geometry}. Summing the fixed finite
collection of subsets proves the raw low bound.

\emph{Poisson representation and normalization.}

The source coefficientwise Poisson identity gives
\begin{equation}\label{eq:high-identity}
 \begin{split}
 \mathscr I_\eta(Z)
 =\frac1{(2\pi i)^3}\int_{(3)}\int_{(3)}\int_{(2)}
 &\mathcal W(X,Y,Z;s,w,z)
 \sum_u^{(6)}(\N u)^{-z}\overline{\xi(u)}\\
 &\cdot
 \frac{\zeta_F^S(6z)L^S(w,\chi_\bullet(u))}
 {L^S(s,\eta\overline{\chi_\bullet(u)})}
       \mathfrak H_{\eta,u,Z}(s,w,z)\,dz\,dw\,ds,
 \end{split}
\end{equation}
where
\[
 \mathcal W=X^{1/2-z}Z^{s+z-1}Y^{w-1}
               e^{(s+z-1)^2}M_0(z)\widehat W_1(w).
\]
The notation \(M_0(z)\) here denotes the source
radial Mellin test, not an inverse polynomial.
It is holomorphic for \(\Rea z>0\).
The identity is absolutely convergent on these
lines and concerns exactly \eqref{eq:physical-probe}.

The full correction is a finite sum of
holomorphic selected local factors and normally
convergent unselected products. It is not
defined by division by raw local factors.
Appendix~\ref{app:contours} verifies its required
regions and the changes below \(7/8\).

Only \(u=1\) has a principal numerator.
For this row put
\[
 H_\eta(s)=\mathcal H_{\eta,1}(s,1,1/6).
\]
A fixed initial-prime cutoff gives
\begin{equation}\label{eq:H-contraction}
 H_\eta\text{ holomorphic on }\Rea s>87/100,
 \qquad |H_\eta(s)-1|\le1/2.
\end{equation}
The scalar normalizer \eqref{eq:normalization}
is applied to both sides of \eqref{eq:high-identity}.
The source fixed-ray prime theorem gives
\[
 A_T(Z)\asymp(\log Z)^{-K}>0,\qquad
 A_T(Z)^{-1}\ll_\eps Z^\eps
\]
for sufficiently large real \(Z\).
Combining this with the raw low bound and
\eqref{eq:signal-low} proves \eqref{eq:normalized-low}.
The Gaussian makes \eqref{eq:principal-signal}
absolutely convergent on its defining line.
\end{proof}

\section{Prime amplitudes and the detected row count}
\label{sec:rows}

This section supplies the number of rows in each class
to the high-side comparison. A class is specified by its
row scale \(U=Z^d\), denominator boundary \(a\), and
prime amplitudes. A large amplitude enlarges the row's
term in the probe but strengthens the amplification used
to count such rows. We keep both effects visible by proving
a bound of the form \(\#\mathcal B\ll
U^{R_\kappa(\delta,x)+\eps}\), up to height factors.
Section~\ref{sec:numerator} uses this count as an input,
and Section~\ref{sec:high} combines it with the row weights.

\subsection{All physical amplitudes are retained}

On a row dyad \(\N u\asymp U=Z^d\), define
\[
 Q_i(u;z)=P_i^{-1/2}
  \sum_{p\in\mathcal P_i(Z)}
       \overline{\chi_p(u)}W_i(\N p/P_i)
                         (\N p/P_i)^{z-1}.
\]
The source pointwise prime estimate and the
buffered family \(\mathcal X_u\) permit amplitude
bins with
\[
 0\le g_i\le\delta/2,\qquad
 |Q_i|\le P_i^{g_i+\eps_{\rm bin}},\qquad
 g_i>0\Longrightarrow |Q_i|\ge P_i^{g_i}.
\]
The height and contour losses are assigned before
these bounds are used. Record \(g_i\) for every
physical slot, even when a later expansion term
uses the local error from that slot.
Put
\begin{equation}\label{eq:amplitude}
 q=\frac{\sum_i\ell_i g_i}{\ell},\qquad
 x=q/\delta,\qquad
 0\le q\le\delta/2,\quad0\le x\le1/2.
\end{equation}
Here \(q\) is the prime-length-weighted mean growth
exponent of the physical prime sums, while \(x\)
measures it relative to the detector parameter
\(\delta\). Both are fixed within an amplitude class.
At base \(U\), the available prime length is
\(\mathscr L=\ell/d\).
Thus \(b,\ell,\kappa\) are fixed for the whole probe,
whereas \(d,\delta,q,x\) describe one row class. The ratio
\(\mathscr L\) converts the physical prime supply from
base \(Z\) to the base \(U\) of the moments.

\subsection{The row-count exponent}

The exponent \(R_\kappa\) results from balancing the short
simultaneous-witness estimate with the long inverse-witness
estimate. It is the output needed later; the functions
\(D_\kappa,P_\kappa,\mathcal J_\kappa\) below are algebraic
abbreviations for that balance, not additional parameters.
Let \(\alpha=5/6\) and \(c=1/(3\kappa)\), and put
\begin{equation}\label{eq:row-functions}
 \begin{split}
 D_\kappa(x)&=3-(1+2c)x,\\
 P_\kappa(x)&=(2-2cx)(1-x),\\
 \mathcal J_\kappa(\delta,x)
       &=(\alpha-\delta)D_\kappa(x)+\delta P_\kappa(x),\\
 R_\kappa(\delta,x)
       &=1-\delta+
         \frac{(\alpha-\delta)\delta P_\kappa(x)}
                  {2\mathcal J_\kappa(\delta,x)}.
 \end{split}
\end{equation}

\begin{proposition}[Detected row count]\label{prop:row-count}
Assume the inverse and plain moments and simultaneous
witnesses of Section~\ref{sec:inputs}, including their
coefficient and uniformity conditions. Let
\(13/18\le\kappa\le1\), with
\(\beta_*\le(1+\kappa)/2\) when \(\kappa<1\).
On a dyad \(\N u\asymp U=Z^d\), let \(\mathcal B\)
be a nonunit physical amplitude class above the
detector floor, with \(\delta=2a-1\), \(x=q/\delta\),
and \(\mathscr L=\ell/d>1/5\) by a fixed amount.
Then
\[
 \#\mathcal B\ll U^{R_\kappa(\delta,x)+\eps}(1+T_1)^A,
 \qquad 0<\delta\le5/6,\quad0\le x\le1/2.
\]
On this rectangle the denominator in
\eqref{eq:row-functions} is positive, and
\begin{equation}\label{eq:coarse-count}
 \frac16\le R_\kappa\le1-\frac{2\delta}{3}\le1.
\end{equation}
For prescribed bounded length ranges and a requested
loss, the slot mesh is fixed before the slot count.
For each fixed slot system and arithmetic datum,
the seminorm and polynomial height orders are finite.
\end{proposition}

\begin{proof}
\emph{Prime selection.}
For \(0\le z_p\le\mathscr L\), a subcollection
of whole positive-amplitude slots can be chosen
with total length at most \(z_p\) and
\[
 |Q_{\rm sel}|^2\ge U^{2qz_p-\eps}.
\]
The selection is fixed within the amplitude bin,
and its loss is controlled by the common slot mesh.
To see this, order the slots by decreasing \(g_i\).
A fractional initial segment of length \(z_p\)
has average amplitude at least the full weighted
mean \(q\). Discard its last fractional slot.
The lost gain is at most \((\delta/2)\) times
the largest relative slot length.
If the positive-amplitude slots are exhausted,
their full gain is \(q\mathscr L\ge qz_p\).
Taking the mesh sufficiently small proves the
claim, including the empty selection.

Relative to a witness character
\(\psi=\nu\overline{\chi_\bullet(u)}\), the
coefficient of a physical prime is
\[
 \overline{\nu(p)}1_{p\in1_T}
  =\frac1{|T|}\sum_{\theta\in\widehat T}
                         (\overline\nu\theta)(p).
\]
This is an allowed fixed combination from \(\Theta\).
When the positive orientation is needed in the
plain moment, conjugate the entire product,
including the profiles and all selected slots.
All absolute values, masks and underlying
supports are unchanged.

\emph{Combining the witnesses.}
Fix a detector parameter \(t\in[1,3/2]\) and
subdivide by the simultaneous witnesses in
\eqref{eq:witnesses}.

\emph{Short witnesses.}
For \(r<1\), the inverse capacity is
\(z_M=(1-r)/2\); for \(m<1/2\), the plain
capacity is \(z_P=(1-2m)/(6\kappa)\).
The selected moments give the respective exponents
\[
 \begin{split}
 A_I(r)&=1-\delta\{x+(1-x)r\},\\
 A_P(m)&=1-\delta\{cx+(2-2cx)m\}.
 \end{split}
\]
Use \(m\ge t-r-o(1)\). The two expressions cross at
\[
 r_*(t)=
 \frac{(2-2cx)t+(c-1)x}{D_\kappa(x)}.
\]
For \(c\le6/13\), direct multiplication by the
positive denominator gives
\[
 3/5\le r_*(t)\le1,\qquad
 1/3\le t-r_*(t)\le1/2.
\]
Use the inverse bound for \(r\ge r_*(t)\)
and the plain bound for \(r\le r_*(t)\).
Their common value is
\[
 R_{\rm sh}(t)=1-\delta+
       \frac{\delta P_\kappa(x)}{D_\kappa(x)}
                                      (3/2-t).
\]
This choice also verifies admissibility.
The inverse selection has length at most \(1/5\).
After decreasing it by \(\nu_0>0\), its first
strict margin is \(2\nu_0\), and its second is
\[
 3-2r-8z_M=2r-1+8\nu_0\ge1/5.
\]
The selected plain length is at most
\(1/(18\kappa)\le1/13\).
Near zero capacity use the appropriate unmarked
estimate; the resulting loss is arbitrarily small.

\emph{Long witnesses and the choice of detector.}
If \(r\ge1\), \eqref{eq:long-inverse} and \(r\le t+o(1)\)
give the exponent
\[
 R_{\rm long}(t)=1-\delta+(\alpha-\delta)(t-1).
\]
If \(m\ge1/2\), the zero-slot plain moment gives
\(1-\delta+o(1)\).
Balancing the remaining two branches at
\[
 t_*=1+
 \frac{\delta P_\kappa(x)}{2\mathcal J_\kappa(\delta,x)}
\]
gives \eqref{eq:row-functions}.
At \(\delta=\alpha\), use the unmarked endpoint
\(t=3/2\); vanishing-capacity neighborhoods are
handled as above. The finite presentation choices,
dyadic subdivisions and source Sobolev bounds
give the stated uniform loss.

\emph{Denominator and coarse bounds.}
Both \(D_\kappa\) and \(P_\kappa\) are positive
on the stated rectangle, so \(\mathcal J_\kappa>0\).
The lower bound follows from \(R_\kappa\ge1-\delta\).
The identity
\[
 2D_\kappa-3P_\kappa=2x(2+c-3cx)\ge0
\]
gives \(P_\kappa/D_\kappa\le2/3\).
Hence
\(R_\kappa\le1-\delta+\delta P_\kappa/(2D_\kappa)
\le1-2\delta/3\).
\end{proof}

\section{An amplified first moment for the numerator}
\label{sec:numerator}

The row count alone does not estimate the numerator
on those rows. We now use the prime amplitudes that
define a row class to improve its numerator average.
The relevant character is the physical numerator
\(\chi_\bullet(u)\), which may differ from the fixed
twist used to detect the row.

The choice of contour can be understood one dyadic
block at a time. The approximate functional equation
at \(1/2+\vartheta+it\) attaches a coefficient
\(D^{-\vartheta}=U^{-\vartheta m}\) to a centrally
normalized block of length \(D=U^m\).
Two copies of this block use length \(2m\) in the
plain fourth moment. For \(m\le1/2\), the remaining
capacity can therefore accommodate prime length
\(z_p=v(1-2m)\), with \(v\le1/(6\kappa)\).
The selected primes give a fourth-moment saving
\(U^{-2qz_p}\), hence a first-moment saving
\(U^{-qz_p/2}\) after H\"older.

Short blocks leave more room for prime amplification;
longer blocks have less room but stronger coefficient
decay. Choosing \(\vartheta=qv\) makes their combined
contribution independent of the block length:
\begin{equation}\label{eq:numerator-balance}
 -\vartheta m-\frac{qv(1-2m)}2
       =-\frac{\vartheta}{2}.
\end{equation}
For \(m\ge1/2\), coefficient decay alone gives this
saving. The parameter \(v\) is the prime-length capacity
reserved for a zero-length block; an actual block uses only
\(v(1-2m)\). Its two possible limits are the physical supply
\(\mathscr L\) and the moment capacity \(1/(6\kappa)\).
The following proposition first respects both limits. Its
second alternative uses full moment capacity even when
the physical supply is smaller: a direct bound handles
the shortest blocks, provided \eqref{eq:short-gate} holds.
The symbol \(R\) in the statement is an input row-count
exponent; the high-side application takes
\(R=R_\kappa(\delta,x)\) from Proposition~\ref{prop:row-count}.

\begin{proposition}[Amplified numerator average]\label{prop:numerator}
Let \(13/18\le\kappa\le1\), and assume the plain
moment \eqref{eq:plain-moment} with its stated
coefficient and uniformity conditions. When
\(\kappa<1\), assume \(\beta_*\le(1+\kappa)/2\).
Let \(\mathcal B\) be a class of nonunit physical
rows \(\N u\asymp U\), with fixed amplitude bins
as in \eqref{eq:amplitude}, total available prime
length \(\mathscr L\), and
\[
 \#\mathcal B\ll U^{R+\eps}(1+H)^A,
\]
where \(1\le H\le U^{1/100}\) bounds the numerator
and physical prime-polynomial heights, and
\(0\le q\le5/12\). Choose either
\begin{enumerate}
\item \(\vartheta=qv\), where
      \(0\le v\le\min\{\mathscr L,1/(6\kappa)\}\); or
\item \(\vartheta=q/(6\kappa)\), provided
\begin{equation}\label{eq:short-gate}
 R\le(6\kappa-2q)\mathscr L.
\end{equation}
\end{enumerate}
Then, uniformly for \(|t|\le H\),
\begin{equation}\label{eq:numerator-moment}
 \sum_{u\in\mathcal B}
 |L(1/2+\vartheta+it,\overline{\psi_u^*})|
 \ll U^{(3R+1)/4-\vartheta/2+\eps}(1+H)^{A'}.
\end{equation}
Here \(\psi_u^*\) induces the physical numerator.
For prescribed bounded length ranges and a requested
loss, the prime-slot mesh can be fixed before the
slot count. Once that count and the arithmetic data
are fixed, the estimate uses finitely many smooth
seminorms and a fixed polynomial height order.
It includes the primitive conductor profiles and
the restoration of the fixed excluded Euler factors.
\end{proposition}

\begin{proof}
\emph{Supply-limited choice.}
Consider a centrally normalized plain block
\(S_m\) of length \(D=U^m\), with the same
character orientation as the physical primes.
For \(0\le m\le1/2\), request
\[
 z_p=v(1-2m).
\]
It fits both the supply and the moment:
\[
 z_p\le\mathscr L,\qquad 2m+6\kappa z_p\le1.
\]
The nonunit row induces outside \(\Theta\).
Apply \eqref{eq:plain-moment} with two copies
of \(S_m\). To select primes, order the slots by
decreasing positive amplitude and take an initial segment
of total length at most \(z_p\). Fractional selection
would give gain at least \(qz_p\); discarding the
last fractional slot loses only the prescribed mesh
error. If the positive-amplitude slots are exhausted,
their total gain is \(q\mathscr L\ge qz_p\).
Thus \(|Q_{\rm sel}|^2\ge U^{2qz_p-\eps}\).
Then
\[
 \sum_{u\in\mathcal B}|S_m(u)|^4
 \ll U^{1-2qz_p+\eps}(1+H)^{A_1},
\]
so H\"older gives
\begin{equation}\label{eq:block-first-moment}
 \sum_{u\in\mathcal B}|S_m(u)|
 \ll U^{(3R+1)/4-qz_p/2+\eps}(1+H)^{A_2}.
\end{equation}
For \(m\ge1/2\) in a fixed bounded range,
the zero-slot moment gives the same estimate
with \(z_p=0\).

The approximate functional equation in
Appendix~\ref{app:afe} expresses the numerator
as sums of such blocks. On each dyad its
main coefficient is \(D^{-\vartheta}\).
The dual coefficient can be written
\[
 \mathcal Q_u^{-\vartheta}D^\vartheta
 =D^{-\vartheta}(D/\sqrt{\mathcal Q_u})^{2\vartheta}.
\]
The last factor belongs to a uniformly controlled
profile. The fixed-prime restorations, whole-factor
conjugations, and rowwise conductor parameters
are justified in that appendix without changing
the permitted coefficient class.
For \(m\le1/2\), \eqref{eq:numerator-balance} gives
the exponent \(-\vartheta/2\).
For \(m\ge1/2\),
\(-\vartheta m\le-\vartheta/2\).
Sum the retained dyads and absorb their logarithmic
number and the arbitrarily small rounding losses.

\emph{Full-capacity choice.}
Only the case \(\mathscr L<1/(6\kappa)\) needs
attention. The moment capacity exceeds the available
prime length only for the shortest blocks. We split
at the length where the two quantities coincide and
use absolute counting below that threshold. Put
\[
 m_0=1/2-3\kappa\mathscr L,\qquad
 \mathcal K(R)=(3R+1)/4.
\]
For \(m_0\le m\le1/2\), the full capacity
\(z_p=(1-2m)/(6\kappa)\) fits the available
prime length, and the preceding proof applies.
For \(0\le m\le m_0\), absolute ideal counting
gives \(|S_m|\ll U^{m/2+\eps}(1+H)^C\).
After the dyadic coefficient, the first-moment
exponent is \(R+(1/2-\vartheta)m\).
Since \(\vartheta<1/2\), it is bounded at \(m_0\);
the exact difference is
\[
 R+(1/2-\vartheta)m_0
        -\{\mathcal K(R)-\vartheta/2\}
   =\frac{R-(6\kappa-2q)\mathscr L}{4}\le0.
\]
The dual profiles satisfy the same absolute bound.
All longer blocks were already covered.
\end{proof}

\begin{remark}
The same prime amplitudes can be used in both the
row count and the numerator moment. They are lower
bounds that hold on every row of the chosen class;
applying two estimates to that class does not insert
two additional prime products into the probe.
\end{remark}

The saving \(U^{-\vartheta/2}\) in
\eqref{eq:numerator-moment} is not yet the net saving in
the probe. Reflection costs \(U^\vartheta\), and the
outside Mellin weight contributes \(Y^{-\vartheta}\).
Relative to the contour \(\vartheta=0\), their product is
\[
 U^{-\vartheta/2}U^\vartheta Y^{-\vartheta}
       =Z^{\vartheta(d/2-l_y)}.
\]
The next section therefore imposes \(h<2l_y\): at the
critical row scale \(d=h\), the averaged numerator then
improves the complete high-side estimate, not just one factor.

\section{The new high-side comparison}\label{sec:high}

We combine the row count and numerator moment with
the Poisson expansion. Moving the numerator variable
left of \(1/2\) reflects its \(L\)-value to the line
\(1/2+\vartheta\), where the preceding estimate
applies. Reflection costs a conductor power, so the
numerator saving must be combined with this cost and
the outside Mellin factors before deciding whether
the new contour improves the bound.

\subsection{From the row estimates to an endpoint condition}

Use \(l_x,l_y,h,B(\ell)\) from Proposition~\ref{prop:low}
and \(R_\kappa\) from Proposition~\ref{prop:row-count}.
Put \(c=1/(3\kappa)\), \(z_0=17/50\),
and \(d_{\min}=1/100\). The fixed line \(\Rea z=z_0\)
is used for intermediate rows; rows below \(Z^{d_{\min}}\)
will instead be estimated absolutely. For a row class with
\(U=Z^d\), \(a=(1+\delta)/2\), and \(q=\delta x\), write
\[
 \mathcal K=\frac{3R_\kappa(\delta,x)+1}{4},\qquad
 \vartheta=
 \begin{cases}
 q\min\{\ell/d,1/(6\kappa)\},&0<\delta<1/8,\\
 q/(6\kappa),&1/8\le\delta\le5/6.
 \end{cases}
\]
Here \(\mathcal K\) is the first-moment exponent obtained
from the row count by H\"older. The contour displacement
\(\vartheta\) uses the two alternatives in
Proposition~\ref{prop:numerator}: a supply-limited choice
for small \(\delta\), and full moment capacity otherwise.
The numerical conditions below will ensure that both choices
are licensed on every intermediate row class.

Before listing those conditions, we record the exponent they
are designed to control. Relative to the direct-bound exponent
\(C_b(B(\ell))\), the contributions are
\[
\begin{array}{c|l}
\text{source}&\text{exponent of }Z\\ \hline
\text{outside Mellin scales}&
 a-B(\ell)-l_y(1/2+\vartheta)+h(z_0-1/6)\\
\text{centered prime factors and amplitudes}&-\ell/2+\ell q\\
\text{numerator average and reflection}&d(\mathcal K+\vartheta/2)\\
\text{Poisson-row weight}&-dz_0.
\end{array}
\]
Their sum defines the total relative exponent
\begin{equation}\label{eq:full-exponent}
 \begin{split}
 \mathcal E(d;\delta,x)
 ={}&a-B(\ell)-l_y(1/2+\vartheta)
             +h(z_0-1/6)-\ell/2+\ell q\\
       &+d(\mathcal K+\vartheta/2-z_0).
 \end{split}
\end{equation}
Sections~\ref{subsec:high-local} and \ref{subsec:high-exponent}
justify this assembly for the full main/error expansion,
including the contour moves and moment hypotheses.
Lemma~\ref{lem:frequency-monotonicity} then shows that
\(\mathcal E\) increases with \(d\), with slope at most
\(2\). Consequently, a bound at \(d=h\) controls all
intermediate scales up to \(h+\zeta\), at a cost of at
most \(2\zeta\). For \(\delta\ge1/8\), that endpoint is
\begin{equation}\label{eq:endpoint}
 \begin{split}
 E_{b,\ell,\kappa}(\delta,x)
 ={}&C_0+\delta/2
 +\delta x\left\{\ell+
       c\left(\frac h4-\frac{l_y}{2}\right)\right\}
 -\frac{3h}{4}(1-R_\kappa(\delta,x)),\\
 C_0={}&-\frac14+\frac b6+\frac{5\ell}{4}.
 \end{split}
\end{equation}
Thus \(E_{b,\ell,\kappa}\) is not an independent input:
it is \(\mathcal E(h;\delta,x)\) in the full-capacity
range. The constant \(C_0\) collects its terms independent
of the row class. Small \(\delta\) has a supply-limited
endpoint, evaluated in Section~\ref{subsec:high-endpoint}.

\subsection{Numerical conditions and the remainder bound}

\paragraph{Admissible parameters.}
The parameters \(b,\ell,\kappa\) determine the probe geometry
and the moment capacity. The remaining three numbers only
record margins: \(g\) is the endpoint saving, \(\zeta\)
extends the retained row range beyond \(h\), and \(g_*\)
is the saving reserved after that extension, before the small
analytic losses. They do not change the arithmetic probe.
Call \((b,\ell,\kappa;g,\zeta,g_*)\) admissible if
the following explicit conditions hold. These are
numerical conditions, not analytic assumptions:
\begin{enumerate}
\item \emph{Direct estimate and contour region.}
      \(b>0\), \(1/6\le\ell\le1/5\),
      \(29991/40000\le\kappa\le3/4\),
      \(B(\ell)>87/100\),
      \(l_x>\ell\), and \(l_y-\ell>11b/6\).
\item \emph{Prime supply and numerator amplification.}
      \(6\kappa\ell>3/4\),
      \(\ell/h<1/(6\kappa)\),
      \(\ell/(h+\zeta)>21/100\),
      and \(h/2<l_y<1/2\).
\item \emph{Intermediate-row saving.}
      \(0<g_*<1/1000\), \(0<2\zeta<1/1000\),
      \(g-2\zeta>g_*\), and
      \(E_{b,\ell,\kappa}(\delta,x)\le-g\)
      for \(1/8\le\delta\le5/6\), \(0\le x\le1/2\).
\item \emph{The remaining row ranges.} With
      \(A_{\rm sm}=1/2+(\ell/2)(1/2+2\ell/h)-h/2\),
      the remaining scalar bounds are
\begin{equation}\label{eq:admissible-ranges}
 \begin{gathered}
 \max\{C_0,C_0+A_{\rm sm}/8\}<-1/250,\\
 C_0+(1/2+3\ell/2)/50<-1/100,\\
 h(z_0-1/6)-l_y/2+2d_{\min}<-1/20.
 \end{gathered}
\end{equation}
      These three inequalities control, respectively, the
      small-\(\delta\) classes above the detector floor,
      the floor itself, and the small rows (including the
      bounded nonprincipal unit rows).
\end{enumerate}

These conditions certify that the row estimates fit together.
They do not supply the family bound
\(\beta_*\le(1+\kappa)/2\), which remains a separate
analytic hypothesis of the following proposition and is
licensed in order in the final two-stage argument.

\begin{proposition}[Uniform high-side remainder]\label{prop:high}
Assume the analytic inputs of Section~\ref{sec:inputs}.
Let \((b,\ell,\kappa;g,\zeta,g_*)\) be admissible,
and suppose \(\beta_*\le(1+\kappa)/2\).
If \(\Delta=\beta_*-B(\ell)>0\), there exist a fixed
even prime-slot system and \(\rho>0\), common to all
targets, such that the probe and principal term of
Proposition~\ref{prop:low} satisfy
\begin{equation}\label{eq:high-remainder}
 |J_\eta(Z)-f_\eta(Z)|
       \ll_\eta Z^{C_b(\beta_*)-\rho}
\end{equation}
for every fixed finite-order target \(\eta\) and all
sufficiently large real \(Z\). Constants and thresholds
may depend on \(\eta\), but \(\rho\) does not.
The same slot system retains the low bound of
Proposition~\ref{prop:low}; it is fixed before \(Z\).
\end{proposition}

The rest of this section proves the proposition.
The local contour calculations used in its proof
are supplied in Appendix~\ref{app:contours}.

\subsection{Local errors with the reflected value retained}
\label{subsec:high-local}

Write \(s\) for the denominator variable and put
\(z_0=17/50\). On a detected class take
\begin{equation}\label{eq:raised-lines}
 \Rea s=a+16e,\qquad
 \Rea w=1/2-\vartheta,\qquad \Rea z=z_0,
 \qquad 0\le\vartheta\le\min\{a-1/2,1/10\}.
\end{equation}
The main factor of slot \(i\) is
\(-P_i^{z-1/2}Q_i(u;z)\); denote its local error
by \(\mathcal D_i(u)\).

\begin{lemma}\label{lem:raised-errors}
For every subset \(I\) of error slots, on
\eqref{eq:raised-lines},
\[
 \left|L^S(w,\chi_\bullet(u))
                 \prod_{i\in I}\mathcal D_i(u)\right|
 \ll U^{\vartheta+\eps}(1+T_1)^C
 |L(1-w,\overline{\psi_u^*})|
                 \prod_{i\in I}P_i^{z_0-1/2+\eps}.
\]
All conductor-deficit factors are used simultaneously.
\end{lemma}

The proof from the exact local formulas is in
Appendix~\ref{app:contours}. Unlike the old
pointwise specialization, this estimate leaves the
reflected \(L\)-value to be averaged.

\subsection{The exponent at an arbitrary row scale}
\label{subsec:high-exponent}

We now verify the moment hypotheses for the proposed
\(\vartheta\) and derive \eqref{eq:full-exponent}.
For \(d_{\min}\le d\le h+\zeta\), admissibility
gives \(\mathscr L=\ell/d>21/100\).
Since \(\kappa\ge29991/40000\), for \(\delta\ge1/8\)
Proposition~\ref{prop:row-count} gives
\begin{equation}\label{eq:uniform-short-gate}
 \begin{split}
 (6\kappa-2q)\mathscr L-R_\kappa
 &\ge\frac{21}{100}(6\kappa-\delta)
                           -(1-2\delta/3)\\
 &\ge\frac{10799}{6000000}>\frac1{600}.
 \end{split}
\end{equation}
Thus the two alternatives in Proposition~\ref{prop:numerator}
apply in their respective ranges.
Moreover \(\vartheta\le q/(6\kappa)\le5/52<1/10\)
and \(\vartheta\le a-1/2\).

The numerator moment contributes
\(U^{\mathcal K-\vartheta/2}\), while the functional
equation and local errors cost \(U^\vartheta\).
Their product is \(U^{\mathcal K+\vartheta/2}\),
up to the small losses and height factors.
The outside factor \(Y^{w-1}\) supplies a further
\(Y^{-\vartheta}\). Thus the net contour-dependent
factor is \(Z^{\vartheta(d/2-l_y)}\), which decays
at the critical scale \(d=h\) because \(h<2l_y\)
by admissibility.

Combining the averaged numerator and
Lemma~\ref{lem:raised-errors}, the full finite
main/error expansion has the common row bound
\[
 U^{\mathcal K+\vartheta/2+\eps}(1+T_1)^A
              Z^{\ell(z_0-1/2)+\ell q}.
\]
Indeed an error term has main amplitude only
\(\sum_{i\notin I}\ell_i g_i\le\ell q\),
whereas every physical amplitude remains
available in estimating that row class.

Multiplying by the outside Mellin factors and the row
weight \((\N u)^{-z_0}\) gives exactly
\eqref{eq:full-exponent}, relative to \(C_b(B(\ell))\).
This verifies the exponent accounting at the start of the
section for every term, not just the product of main factors.
Appendix~\ref{app:contours} justifies the
amplitude-dependent \(w\)-line before the local
main/error bounds are applied.

\begin{lemma}\label{lem:frequency-monotonicity}
For every admissible parameter tuple,
with \(\delta,x\) fixed, \(\mathcal E\) is continuous
and increasing in \(d\), with one-sided slopes
between \(7/200\) and \(2\).
\end{lemma}

\begin{proof}
By Proposition~\ref{prop:row-count},
\(\mathcal K-z_0\ge7/200\).
Where \(\vartheta=q/(6\kappa)\), the derivative is
\[
 \mathcal K-z_0+\frac{q}{12\kappa}.
\]
Where \(\vartheta=q\ell/d\), it is
\[
 \mathcal K-z_0+\frac{q\ell l_y}{d^2}.
\]
The two choices of \(\vartheta\) coincide at
\(d=6\kappa\ell\), so \(\mathcal E\) is continuous
there; its one-sided derivatives need not agree.
The second regime has \(d\ge6\kappa\ell>3/4\)
by admissibility.
Using \(R_\kappa\le1\), \(q\le5/12\),
\(\ell\le1/5\), and \(l_y<1/2\) bounds both
derivatives above by \(2\).
\end{proof}

\subsection{The endpoint}
\label{subsec:high-endpoint}

Monotonicity reduces the intermediate row scales
to their upper endpoint. We first evaluate at
\(d=h\); the small extension beyond \(h\) will
be absorbed by the strict margins in the certificates.
A negative endpoint exponent means that these rows
are smaller than the low bound by a fixed power of \(Z\).

For \(\delta\ge1/8\), evaluation of
\eqref{eq:full-exponent} at \(d=h\) gives exactly
the endpoint function \eqref{eq:endpoint}.
For \(0<\delta<1/8\), admissibility gives
\(\ell/h<1/(6\kappa)\), and the endpoint instead is
\begin{equation}\label{eq:small-delta-endpoint}
 E_{\rm small}(\delta,x)
 =C_0+\delta/2+
       \ell\delta x(1/2+2\ell/h)
       -\frac{3h}{4}(1-R_\kappa).
\end{equation}
The two expressions differ only in how much prime
length is used for the numerator. For small \(\delta\)
we cap the selection by the available supply; for
larger \(\delta\) the short-block bound permits the
full moment capacity.

\subsection{Proof of the main estimate}

\begin{proof}[Proof of Proposition~\ref{prop:high}]
The licensing hypothesis and Lemma~\ref{lem:kappa-extension}
make the plain moment available. Propositions~\ref{prop:row-count}
and \ref{prop:numerator} therefore apply with the
amplitude-dependent contours just established.

For intermediate rows with \(\delta\ge1/8\),
admissibility and Lemma~\ref{lem:frequency-monotonicity}
give \(\mathcal E\le-g+2\zeta<-g_*\).
For \(0<\delta<1/8\), the coarse bound in
Proposition~\ref{prop:row-count} gives
\(E_{\rm small}\le C_0+A_{\rm sm}\delta\).
The first inequality of \eqref{eq:admissible-ranges}
and the extension cost \(2\zeta<1/1000\) leave
a saving greater than \(g_*\).
At the detector floor the original pointwise contour
has endpoint \(C_0+(1/2+3\ell/2)/50\) and slope
less than \(2\), so the second inequality does the same.

For rows below \(Z^{d_{\min}}\), including the
nonprincipal unit rows, the absolute numerator bound
and full correction give exponent at most
\(C_b(\beta_*)+h(z_0-1/6)-l_y/2+2d_{\min}\).
The last inequality of \eqref{eq:admissible-ranges}
gives a saving of \(1/20\). Above \(Z^{h+\zeta}\),
the absolute \(z\)-line can be chosen sufficiently
far right to give any prescribed saving. The
underlying small-row and distant-row estimates are
recorded in Appendix~\ref{app:ranges}.

For the principal row, Lemma~\ref{lem:principal-remainder}
gives the three positive margins \(l_y/20\),
\(h/600\), and \(\mu\), where
\(0<\mu<(87/100)\min_i\ell_i\).
These complete the estimates for every term of the
Poisson representation in Proposition~\ref{prop:low}.

Choose the moment losses and common slot mesh first,
then a fixed even slot system fine enough for that
mesh. Set
\[
 \mu=\frac{87}{200}\min_i\ell_i,\qquad
 \gamma=\frac1{100}
       \min\{\Delta,g_*,l_y/20,h/600,\mu\}.
\]
Choose the detector width and remaining real losses
within these margins. For each target, fix the
required seminorm and height orders before choosing
the height exponent \(\tau>0\) small enough to absorb
their polynomial cost. Choose the external tail order
last. Appendix~\ref{app:choices} records this order
in detail. It yields \eqref{eq:high-remainder} with
\(\rho=2\gamma>0\), common to all targets.
The low bound holds for every fixed slot system,
so it holds for this same choice.
\end{proof}

\section{Two certified parameter choices}\label{sec:parameters}

This section supplies the numerical input to
Proposition~\ref{prop:high}. Its conclusion asserts
admissibility only; the family zero bound needed to
license each moment is supplied in the next section.

\begin{proposition}[Parameter certificate]\label{prop:parameters}
The two triples
\begin{equation}\label{eq:stage-zero}
 \kappa_0=\frac34,\qquad
 b_0=\frac{173}{2000},\qquad
 \ell_0=\frac{10027}{60000}.
\end{equation}
and
\begin{equation}\label{eq:stage-one}
 \kappa_1=\frac{29991}{40000},\qquad
 b_1=\frac{17307}{200000},\qquad
 \ell_1=\frac{1253399}{7500000}.
\end{equation}
are admissible in the sense of Section~\ref{sec:high}
with the respective margins
\[
 \begin{array}{c|ccc}
 j&g_j&\zeta_j&g_{*,j}\\ \hline
 0&10^{-6}&10^{-7}&1/2000000\\
 1&10^{-7}&10^{-9}&1/20000000.
 \end{array}
\]
Their boundaries are
\[
 B(\ell_0)=B_0=\frac{69991}{80000},\qquad
 B(\ell_1)=\Bnew=\frac{8748867}{10000000},
\]
and their licensing identities are
\begin{equation}\label{eq:license}
 \frac{1+\kappa_0}{2}=\frac78,\qquad
 \frac{1+\kappa_1}{2}=B_0.
\end{equation}
\end{proposition}

\begin{proof}
The balanced lengths are
\[
 l_{x,0}=\frac{44783}{120000},\qquad
 l_{y,0}=\frac{55163}{120000},\qquad
 h_0=\frac{31757}{40000},
\]
and
\[
 l_{x,1}=\frac{11195177}{30000000},\qquad
 l_{y,1}=\frac{13791227}{30000000},\qquad
 h_1=\frac{7939473}{10000000}.
\]
Substitution gives the stated boundaries and
licensing identities, as well as
\[
 C_{b_1}(\Bnew)=\frac{2906963}{15000000}.
\]
The only inequalities over a continuous parameter
rectangle are the endpoint bounds. Clearing their
positive denominators gives the exact identities
\eqref{eq:certificate-template}. The integer
coefficients in Appendix~\ref{app:certificates}
have \(\mathcal A_j(y)>0\) and
\(4\mathcal A_j(y)\mathcal C_j(y)-\mathcal B_j(y)^2>0\)
for \(y=1/2-x\ge0\). Completing the square proves
\[
 E_{b_j,\ell_j,\kappa_j}(\delta,x)<-g_j
 \quad(0\le\delta\le5/6,\ 0\le x\le1/2).
\]
All remaining conditions are rational comparisons.
Appendix~\ref{app:ranges} records their exact values:
the low-scale margins are positive, the prime supply
exceeds \(21/100\), and the three remaining scalar
bounds are strictly below \(-1/250\), \(-1/100\),
and \(-1/20\), respectively. In the small-\(\delta\)
bound \(A_{\rm sm}>0\), so its maximum is at
\(\delta=1/8\). Finally \(g_j-2\zeta_j>g_{*,j}\)
and both \(g_{*,j}\) and \(2\zeta_j\) lie in
\((0,1/1000)\). This checks every admissibility
condition in Section~\ref{sec:high}.
\end{proof}

\section{Closing the argument}\label{sec:proof}

Only the conclusions of the probe, high-side, and
parameter propositions are needed here. The following
lemma turns a licensed admissible parameter choice
into a zero-free boundary; no local row estimates
enter its proof.

\begin{lemma}[Closure for a licensed parameter choice]\label{lem:closure}
Assume the analytic inputs of Section~\ref{sec:inputs}.
If \((b,\ell,\kappa;g,\zeta,g_*)\) is admissible
and \(\beta_*\le(1+\kappa)/2\), then
\(\beta_*\le B(\ell)\).
\end{lemma}

\begin{proof}
Suppose \(\Delta=\beta_*-B(\ell)>0\).
Proposition~\ref{prop:high} supplies one fixed
slot system and a saving \(\rho>0\) common to all
targets. For the same probe, Proposition~\ref{prop:low}
with loss \(\Delta/4\) gives
\begin{equation}\label{eq:final-comparison}
 |J_\eta(Z)|\ll_\eta Z^{C_b(B(\ell))+\Delta/4},\qquad
 |J_\eta(Z)-f_\eta(Z)|\ll_\eta Z^{C_b(\beta_*)-\rho}.
\end{equation}
Since \(C_b\) has slope one, it follows that
\[
 |f_\eta(Z)|\ll_\eta Z^{C_b(\beta_*)-\epsilon_*},
 \qquad \epsilon_*=
       \min\{3\Delta/4,\rho,\Delta/2\}>0.
\]
In particular \(\epsilon_*<\Delta\), so
\(\beta_*-\epsilon_*>B(\ell)>87/100\).

The integral defining \(f_\eta\) in
Proposition~\ref{prop:low} can be moved from \(s=2\)
to any larger fixed real part. Absolute convergence
of the reciprocal and Gaussian decay then give
arbitrarily rapid power decay as \(Z\downarrow0\).
Consequently
\[
 \int_0^\infty f_\eta(Z)Z^{-C_b(s)}\,\frac{dZ}{Z}
\]
is holomorphic on \(\Rea s>\beta_*-\epsilon_*\).
Fourier inversion on \(\Rea s=2\) identifies this
transform with
\[
 e^{(s-5/6)^2}\frac{H_\eta(s)}{L_F^S(s,\eta)}.
\]
The identity theorem extends the identity, and
\(|H_\eta-1|\le1/2\) from Proposition~\ref{prop:low}
permits division by \(H_\eta\) throughout the region.
Thus every reciprocal \(1/L_F^S(s,\eta)\) extends
holomorphically to \(\Rea s>\beta_*-\epsilon_*\).
The displacement \(\epsilon_*\) is common to all
targets. By the definition of \(\beta_*\), some
target has a zero to the right of that line, where
its reciprocal has a pole. This is the contradiction.
\end{proof}

\begin{proof}[Proof of Theorem~\ref{thm:main}]
Proposition~\ref{prop:parameters} supplies the two
admissible choices and their licensing identities.
The assumed \(7/8\) half-plane licenses the first
choice, so Lemma~\ref{lem:closure} gives
\(\beta_*\le B_0\). This licenses the second choice;
a second application of the same lemma gives
\(\beta_*\le\Bnew\).
Deleting or restoring the finitely many Euler
factors of an induced character changes no zero
in this half-plane.
The norm-pullback transfer
\cite[Proposition 11.3]{OAI26}, including the
principal case at \(s=1\), gives the Dirichlet
statement.
\end{proof}

Both stages compare the low and high
descriptions of their own physical probe,
with the same scalar normalizer in those
two descriptions. They do not identify the
probe or normalizer from one stage with the
one from the other.

\section{A consequence and the scope of the method}
\label{sec:consequence}

The zero-free conclusion also bounds smoothed
M\"obius sums. The connection is direct: their
Dirichlet series is the reciprocal of the target
\(L\)-function, so a larger zero-free half-plane
allows Mellin inversion to be shifted farther left.

\begin{corollary}\label{cor:mobius}
Under the hypotheses of Theorem~\ref{thm:main},
let \(\eta\) be a fixed finite-order Hecke character
over \(F\), and let \(W\) be smooth and compactly
supported in \((0,\infty)\). For every \(\eps>0\),
\[
 \sum_{\mathfrak a}\mu(\mathfrak a)\eta(\mathfrak a)
                 W(\N\mathfrak a/X)
       \ll_{\eta,W,\eps}X^{\Bnew+\eps}.
\]
\end{corollary}

\begin{proof}
Mellin inversion expresses the sum as the
integral of
\(\widehat W(s)X^s/L_F(s,\eta)\)
on an absolute line.
Theorem~\ref{thm:main} permits its shift to
\(\Rea s=\Bnew+\eps\).
The logarithmic-control estimate of \cite[Section~4]{OAI26} gives
a fixed polynomial height bound for the
reciprocal on this line and on the intervening
strip. The Mellin transform of \(W\) has
arbitrary decay. Horizontal joins therefore
vanish, and the shifted integral has the
stated bound. The reciprocal of a principal
pole has a zero, not a pole.
\end{proof}

The two changes responsible for the final
comparison are the low-side decay retention
and the averaged numerator. The last fixed
parameter contributes only the additional
decrease
\[
 B_0-\Bnew=\frac8{10000000}.
\]
No claim is made that the resulting exponent
is optimal among all variants of the method.

\section*{AI use disclosure}

Generative AI was used interactively to read and understand
the framework of \cite{OAI26}. A self-developed math agentic system Dasheng was used in subsequent mathematical
exploration, including the development, testing, and refinement
of parts of the construction and proof presented here.
The idea of the low-side refinement---retaining the external
decay factor in the direct estimate to permit longer prime
windows---is due to the author.
The resulting zero-free conclusions are conditional on the
correctness of the analytic inputs from \cite{OAI26}, whose
proofs have not been independently verified here.



\clearpage
\appendix

\section{The extension of the plain-moment range}
\label{app:kappa}

This extension is needed only for the second stage.
The issue is whether lowering \(\kappa\) preserves
the margins in the source induction, not whether a
different moment estimate can be proved. We check
the parameter-dependent steps in their order of use.

\begin{proof}[Proof of Lemma~\ref{lem:kappa-extension}]
Retain the full source statement, including
its moving-support and common-puncture
conventions, as needed in its induction.
Write \(M\) for the effective width and
\(A=m_1+m_2+z\).
We check the occurrences of \(\kappa\) in
the proof of \cite[Lemma 18.1]{OAI26}.

For \(\kappa<1\), the hypothesis
\(\beta_*\le(1+\kappa)/2\) licenses the
global logarithmic-derivative bound on
\(\Rea s=(1+\kappa)/2+e\).
Mellin inversion of the von Mangoldt sum,
removal of higher prime powers, and division
by the logarithmic weight give
\[
 |Q|^2\ll Z^{\kappa z+\eps}(1+H)^A.
\]
Each resulting slot character is nonprincipal:
otherwise the inducing row character would
belong to \(\Theta\).
For \(\kappa=1\), absolute counting suffices.
The distance from the family zero boundary
is at least \(e\), uniformly on the compact
admissible parameter interval.

The affine premise and \(A\ge z\) imply
\[
 \kappa z\le M/6,\qquad z\le3M/13
       \qquad(\kappa\ge13/18).
\]
The source width-floor proof uses only the
first inequality for its extra terminal
exponent \(\rho/6\).
The old auxiliary upper bound \(2M/9\)
is replaced by \(3M/13\).

In the centered range \(A>5M/6\),
\[
 (6\kappa-1)z<M/6,\qquad
 6\kappa-1\ge10/3,\qquad z<M/20.
\]
With the same comparison length \(L=M/4\),
\[
 A-z>47M/60>2L.
\]
The reflected comparison therefore satisfies
\[
 A_{\rm comp}<23M/30+\xi,\qquad
 A_{\rm comp}+(6\kappa-1)z<14M/15+\xi.
\]
The two required comparison boundaries
retain their \(M/15\) reserve.

The finite Fourier transformations, their
coefficient identities, and exceptional-row
cancellation do not depend on \(\kappa\).
For a nonexceptional child with positive
affine defect \(F\), removing a slot of
length \(s\) decreases the affine expression
by \(6\kappa s\) and costs \(\kappa s\).
The same greedy prefix with mesh \(\eta\)
has total cost at most
\[
 \kappa\{F/(6\kappa)+\eta\}
       \le F/6+\eta.
\]
If the slots disappear, the completed
unrestricted zero-slot assertion applies.
The coefficient class is unchanged.

Finally \(0\le6\kappa-1\le5\), so the aggregate
support error is still bounded by
\[
 |\Delta M|+|\Delta A|+5|\Delta z|.
\]
These checks cover the parameter-dependent
prime estimate, terminal cap, centered
comparison, child repair and support errors.
The source strict width decrease is unchanged.

Run the same finite induction: first all
zero-slot width bands, then the positive-slot
bands. The width floor, width step, support
allowance and mesh are chosen from the
bounded real ranges and final loss before
the slot count. Their terminal and edge
reserves are exactly the ones just checked.
For each fixed slot system and arithmetic
datum, choose finite seminorm and height
orders backwards through that bounded depth.
This proves the extension, with the same
slot-count-independent mesh and the same
fixed-data uniformity.
\end{proof}

\section{Approximate functional equations and uniform profiles}
\label{app:afe}

We give the analytic details used in
Proposition~\ref{prop:numerator}.
The block-length balance was established in the
main text. Here we justify three features needed
to apply it to actual \(L\)-values: the dual sum has
the same useful dyadic coefficient, restoring fixed
Euler factors preserves the allowed coefficients,
and row-dependent conductors can be handled uniformly.
For a primitive nonprincipal finite-order
character \(\chi\), set
\[
 \mathcal Q=\frac{3\N\mathfrak f_\chi}{4\pi^2}.
\]
The primitive functional equation, in the
normalization recalled in \cite{OAI26}, is
\[
 \Lambda(s,\chi)=\mathcal Q^{s/2}\Gamma(s)L(s,\chi),
 \qquad
 \Lambda(s,\chi)=\epsilon_\chi
                         \Lambda(1-s,\overline\chi),
 \qquad|\epsilon_\chi|=1.
\]
Finite order gives trivial infinite type.
The completed nonprincipal function is entire.

For \(w=\sigma+it\), \(2/5\le\sigma\le3/5\),
put
\[
 V_w(x)=\frac1{2\pi i}\int_{(2)}
 e^{v^2}\frac{\Gamma(w+v)}{\Gamma(w)}
                    x^{-v}\,\frac{dv}{v}.
\]
Move the line in
\(\int\Lambda(w+v,\chi)e^{v^2}\,dv/v\)
from \(2\) to \(-2\).
Only \(v=0\) is crossed. The functional
equation and \(v\mapsto-v\) give
\begin{equation}\label{eq:afe}
 \begin{split}
 L(w,\chi)={}&
 \sum_{\mathfrak n}
 \frac{\chi(\mathfrak n)}{(\N\mathfrak n)^w}
           V_w(\N\mathfrak n/\sqrt{\mathcal Q})\\
 &+\epsilon_\chi\mathcal Q^{1/2-w}
       \frac{\Gamma(1-w)}{\Gamma(w)}
 \sum_{\mathfrak n}
 \frac{\overline\chi(\mathfrak n)}
      {(\N\mathfrak n)^{1-w}}
           V_{1-w}(\N\mathfrak n/\sqrt{\mathcal Q}).
 \end{split}
\end{equation}
Gaussian decay and the source Hecke growth
bounds justify the horizontal limits.
The two Dirichlet series are expanded on
absolute lines; the weighted sums in
\eqref{eq:afe} are absolutely convergent.

\begin{lemma}\label{lem:afe-profiles}
For fixed nonnegative integers \(j,N\),
uniformly for \(2/5\le\sigma\le3/5\),
\[
 |(x\partial_x)^jV_w(x)|
 \ll_{j,N}
 \left(1+\frac{x}{1+|t|}\right)^{-N}.
\]
For \(\sigma=1/2+\vartheta\),
\(0\le\vartheta\le1/10\), the Euler
seminorms of
\[
 \frac{\Gamma(1-w)}{\Gamma(w)}
                   x^{2\vartheta}V_{1-w}(x)
\]
obey the same type of bound after increasing
the chosen decay order.
\end{lemma}

\begin{proof}
For \(x\ge1+|t|\), move the defining line
to a sufficiently large fixed positive
real part and apply Stirling.
For \(x\le1+|t|\), move it to \(-1/4\);
there is no gamma pole in this move, since
\(\sigma\ge2/5\), and the only residue is
the constant \(1\).
Euler derivatives insert powers of \(v\),
not powers of \(t\).
The Gaussian makes their integrals
uniformly convergent, including when
\(t+\Im v\) is bounded.
The resulting powers are those of
\(x/(1+|t|)\), which proves the first claim.
Stirling also gives
\[
 |\Gamma(1-w)/\Gamma(w)|
                \ll(1+|t|)^{-2\vartheta}.
\]
This cancels the possible factor
\((1+|t|)^{2\vartheta}\) in
\(x^{2\vartheta}V_{1-w}(x)\).
The product rule proves the derivative
statement uniformly on the indicated
compact range of \(\vartheta\).
\end{proof}

\paragraph{Primitive and natural presentations.}
At each good prime dividing a physical
sixth-power-free row, its primitive
numerator is ramified and has value zero.
Thus only the fixed excluded primes need
to be restored when using \eqref{eq:afe}.
Write \(\mathfrak n=\mathfrak a\mathfrak b\),
where \(\mathfrak a\) is supported on \(S\)
and \((\mathfrak b,S)=1\).
Take triangle over \(\mathfrak a\) before
applying a moment, so its row-dependent
unit phase has no effect.

For \(\N\mathfrak b\asymp D\), the first
sum has scalar
\((\N\mathfrak a)^{-\sigma}D^{-\vartheta}\)
times a central plain block.
For the dual sum,
\[
 (\N\mathfrak a)^{\sigma-1}
       \mathcal Q^{-\vartheta}D^\vartheta
 =(\N\mathfrak a)^{-\sigma}D^{-\vartheta}
       (\N\mathfrak a D/\sqrt{\mathcal Q})^{2\vartheta}.
\]
The coefficients in both expansions have
bounded total mass, at most
\(\prod_{p\in S}(1-(\N p)^{-1/2})^{-1}\).
The remaining variable has the original
fixed-prime mask, and the good ramified
zeros are unchanged.
The factors \((\N\mathfrak b/D)^{\pm it}\)
are separated norm twists.
Conjugate the whole dual plain factor when
needed, and then the entire moment product
if needed for the source orientation.
All fixed-ray coefficients remain in
\(\Theta\).

\paragraph{Conductor-dependent profiles.}
For a fixed dyad and selected prime product,
temporarily let \(\mathcal Q\) vary as an
auxiliary real profile parameter while
the row character is held fixed.
Apply one-dimensional Sobolev on unit
intervals of \(\log\mathcal Q\) to
\[
 F_u(\mathcal Q)=S_{m,\mathcal Q}(u)^2
                                     Q_{\rm sel}(u).
\]
Its derivative is
\[
 2S_{m,\mathcal Q}(u)
       \partial_{\log\mathcal Q}S_{m,\mathcal Q}(u)
                                     Q_{\rm sel}(u).
\]
The squared integrals of both expressions
are instances of the source moment with
two, possibly different, plain profiles
of the same length. The selected primes
do not depend on the conductor parameter.
Lemma~\ref{lem:afe-profiles} supplies the
uniform Euler derivatives.
Root numbers and global conductor phases
have already been removed by absolute
values. There are \(O(\log U)\) intervals
for \(\mathcal Q\ll_S U\).
Thus arbitrary rowwise conductors are
permitted with only a logarithmic loss.

\paragraph{Tail and height order.}
It suffices to retain dyads \(1\le D\le U^2\).
The height choice has \(H\le U^{1/100}\).
Above that range, the weight argument
divided by \(1+H\) is at least a fixed
multiple of \(U^{149/100}\).
Arbitrary decay, the absolute ideal count,
the \(O(U)\) row count, and the summable
fixed-prime coefficients make the omitted
dyads smaller than any prescribed negative
power of \(U\).
The retained moment lengths lie in
\([0,2]\).
All needed moment and profile orders are
finite and fixed before the late height
exponent. A later external contour-tail
order does not differentiate the
arithmetic moment bound.

\section{Euler factors and the contour arguments}
\label{app:contours}

There are two distinct contour issues. On nonprincipal
rows, the reflected numerator must remain available
for averaging while local errors are estimated.
On the principal row, the residue calculation must
remain valid slightly below \(7/8\). We treat these
separately, always starting from the same full
correction attached to the physical probe.

\subsection{The full correction}

We retain the exact local functions \(H_p,G_p\)
of \cite[Sections 7 and 16]{OAI26}.
For the finite selected tuple the correction is
\[
 \mathfrak H_{\eta,u,Z}
 =\sum_{(p_i)\in\prod_i\mathcal P_i(Z)}
 \prod_i\{W_i(\N p_i/P_i)(\N p_i)^{z-1}G_{p_i}\}
 \prod_{\substack{p\notin S\\p\notin\{p_i\}}}H_p.
\]
The local selected operation extracts the same
scalar Euler factor as the unselected prime.
This proves \eqref{eq:high-identity} directly
from the finite expression
\eqref{eq:physical-probe}.
The proof is coefficientwise in the scales
and does not require the source numerical
choice \(b=1/8,\ell=1/6\).

The source first Euler region is
\[
 \Rea s\ge51/100,\quad
 \Rea z\ge17/50,\quad
 \Rea w\ge-1/100,\quad
 \Rea(s+w)\ge1+\epsilon_0.
\]
There the full correction is holomorphic
and has an all-height majorant by a fixed
power of \(Z\) and the heights, for polynomially
bounded row norms.
Only on a region where \(H_p\) is bounded
away from zero do we use \(G_p/H_p\) to
estimate individual main and error factors.

\subsection{The raised numerator line}

\begin{proof}[Proof of Lemma~\ref{lem:raised-errors}]
On \eqref{eq:raised-lines}, enlarge the fixed
initial-prime cutoff so that every selected
\(H_p\) is bounded away from zero.
For \(p\nmid u\), the normalized error has
the four bounds
\[
 Q^{-\Rea s},\quad Q^{-6z_0},\quad
 Q^{4-5\Rea s-6z_0},\quad
 Q^{1-\Rea w-6z_0},\qquad Q=\N p.
\]
Each is at most \(Q^{-51/100}\).
Counting primes in the slot gives
\(P_i^{z_0-51/100+\eps}\).

For \(p\mid u\), write \(j=v_p(u)\).
After separating \(Q^{z_0}\), the
non-strict boundary exponents are
\[
 \begin{array}{c|l}
 j&\text{exponents}\\ \hline
 1&-1-\Rea w\\
 2&1/2-2\Rea s\\
 3&1-2\Rea s-\Rea w,\quad1-3\Rea s\\
 4&3/2-3\Rea s-\Rea w\\
 5&2-5\Rea s.
 \end{array}
\]
They are strictly below \(-1/2\) because
\(\Rea w\ge1-a\) and \(a\ge51/100\).
The geometric tails are smaller.

The remaining strict term, present for
\(j\ge2\), has exponent
\(-\Rea w=\vartheta-1/2\).
For the complete set \(J_0\) of selected
strict ramified labels in one error tuple,
the primitive numerator conductor satisfies
\[
 \N\mathfrak f_u\ll_S \N u
       \prod_{p\in J_0}(\N p)^{-(j_p-1)}.
\]
The functional equation contributes
\((\N\mathfrak f_u)^\vartheta
 |L(1-w,\overline{\psi_u^*})|\)
and a polynomial height factor.
At each label,
\[
 -\Rea w-(j_p-1)\vartheta
       =-1/2-(j_p-2)\vartheta\le-1/2.
\]
This is one simultaneous conductor-deficit
bound, not a repeated use of the same saving.
The labels are distinct because the physical
windows are disjoint.
Summing the divisor-many ramified labels
costs \(U^\eps\).
The other deleted Euler factors are harmless,
since \(\Rea w\ge2/5\).
The reflected \(L\)-value has been retained,
not replaced by a pointwise subpower bound.
\end{proof}

\subsection{Amplitude-dependent contours}

Work first on one finite intermediate row
dyad. Move \(z\) from its absolute line to
\(z_0\), keeping \(s,w\) on their right-hand
lines. No scalar pole is crossed.
Truncate the three original heights before
lowering \(s\); on the exterior the
reciprocal stays global or absolute.
The source external integrated and trace-tail
lemma gives
\[
 O_{\eta,N}(Z^{B_\eta}T_1^{-N}),
\]
where \(B_\eta\) is fixed before \(N\).

Move the retained \(s\)-segment to \(a+16e\)
while \(w\) remains on the right.
For each fixed retained \(z\), partition
the finite row set by the full amplitude
vector and move \(w\) to its assigned line.
The polynomials \(Q_i(u;z)\) are independent
of \(w\), so that row set is fixed during
the move.
They are continuous in \(z\); deterministic
rounding of their absolute values gives
Borel partitions.
The strip stays in the first Euler region,
the numerator is entire on these rows,
and the denominator is independent of \(w\).

On a \(w\)-join, keep the \(s\)-integration
inside its buffered height interval.
Use the global numerator bound and the full
correction majorant, together with the
buffered reciprocal on that domain.
The domain-restricted trace estimate
therefore gives the same error.
No merely buffered \(s\)-axis is extended
to all heights.
All original and detector frequencies
receive fixed shares of their single
cumulative allowance.
The approximate functional equation uses
global identities and introduces no new
buffered reciprocal shift.
Fubini on the finite retained boxes and
the uniform tail bounds justify the
remaining \(z\)-integration.
There is no further \(z\)-contour change
after the amplitude partition.

\subsection{The principal region below \(7/8\)}

The second Euler region extends to
\[
 \mathcal D_2'=
 \{\Rea s\ge87/100,\quad
   \Rea w\ge19/20,\quad
   \Rea z\ge33/200\}.
\]
In the notation of the exact source local
formulas,
\[
 |R|\le Q^{-221/100},\quad
 |V|\le Q^{-99/100},\quad
 |D|\le Q^{-87/100}.
\]
Their denominators \(1-R,1-V,1-D\) are
nonzero. Substitution gives
\[
 H_p-1=
 \begin{cases}
 O(Q^{-181/100}),&p\nmid u,\\
 O(Q^{-41/50}),&p\mid u.
 \end{cases}
\]
The first product converges normally;
the second is finite with its conductor
subpower bound. The selected factors
\(G_p\) have the same nonzero denominators.
Hence the full correction is holomorphic
on this larger region, with the required
all-height majorants.

For \(u=1\), the convergent tail of the
product is uniformly close to \(1\) after
a fixed initial-prime cutoff. This proves
\eqref{eq:H-contraction}.
The principal selected-error exponents are
at most
\[
 -87/100,\quad-99/100,\quad-134/100,\quad-94/100.
\]
Thus
\(\mathcal B_p=-1+O(Q^{-87/100})\).
At \(w=1,z=1/6\),
\[
 \mathfrak H_{\eta,1,Z}(s,1,1/6)
 =H_\eta(s)Z^{\ell/6}A_T(Z)
                       (1+\mathcal R_{\eta,Z}(s)),
\]
where, for any fixed
\begin{equation}\label{eq:principal-mu}
 0<\mu<(87/100)\min_i\ell_i,
\end{equation}
\(|\mathcal R_{\eta,Z}(s)|\ll_\eta Z^{-\mu}\)
uniformly on the global \(s\)-line.
Nonnegative weights make the relative
slot error bounded by \(P_i^{-87/100}\);
the fixed product of slots gives the claim.

\begin{lemma}\label{lem:principal-remainder}
Suppose \(\beta_*>B(\ell)>87/100\).
Let \(\mathscr P_\eta\) denote the \(u=1\)
term of \eqref{eq:high-identity}, and define
\[
 f_\eta(Z)=\frac1{2\pi i}\int_{(2)}
 Z^{C_b(s)}e^{(s-5/6)^2}
                  \frac{H_\eta(s)}{L_F^S(s,\eta)}\,ds.
\]
Then
\[
 \begin{split}
 \left|\frac{\mathscr P_\eta(Z)}{c_S A_T(Z)}
                       -f_\eta(Z)\right|
 \ll{}&
 Z^{C_b(\beta_*)+(1+h)e-l_y/20+\eps}\\
 &+Z^{C_b(\beta_*)+e-h/600+\eps}
 +Z^{C_b(\beta_*)+e-\mu+\eps}.
 \end{split}
\]
\end{lemma}

\begin{proof}
The original source interface states its
lower endpoint at least \(7/8\), so we
repeat its contour proof on \(\mathcal D_2'\).
Keep \(\Rea s=\beta_*+e\).
Move \(w\) past its simple pole at \(1\)
to \(19/20\). In that residue move \(z\)
past \(1/6\) to \(33/200\).
The reciprocal remains to the right of
\(\beta_*\), and \(M_0(z)\) is holomorphic
for \(\Rea z>0\).
The two scalar poles are the only ones
crossed; their product is \(c_S>0\).

The principal correction is bounded by
\(Z^{\ell\Rea z+\eps}\) with polynomial
height cost. The outside real power is
\[
 l_x/2+\Rea s-1+h\Rea z+l_y(\Rea w-1).
\]
It gives the first two displayed remainders.
The error \(\mathcal R_{\eta,Z}\) gives
the third on its original global line.
Move only the main double residue right
to \(s=2\); the reciprocal is holomorphic
there and the Gaussian removes the
horizontal joins.
This is exactly \(f_\eta\), with no
additional outside Euler factor.
\end{proof}

\section{Exact endpoint certificates}\label{app:certificates}

We clear the positive denominator in the endpoint
inequality and obtain a quadratic in \(\delta\).
Its positivity follows by completing the square:
the leading coefficient and \(4\mathcal A_j\mathcal C_j
-\mathcal B_j^2\) both have positive coefficients
as polynomials in \(y=1/2-x\ge0\).
Thus the integer lists below certify the full
parameter rectangle, not just a sampled grid.

For stage \(j=0,1\), let \(E_j\) be
\eqref{eq:endpoint} at its stated parameters.
Put \(y=1/2-x\).
The positive denominator is
\[
 \mathcal J_{\kappa_j}
 =(\alpha-\delta)\{5/2-c_j+(1+2c_j)y\}
       +\delta\{1-c_j/2+2y+2c_jy^2\}.
\]
The following are exact polynomial identities:
\begin{equation}\label{eq:certificate-template}
 2\Lambda_j\mathcal J_{\kappa_j}(-E_j-g_j)
       =\mathcal A_j(y)\delta^2
                    +\mathcal B_j(y)\delta+\mathcal C_j(y).
\end{equation}

\subsection{The first stage}

At \eqref{eq:stage-zero}, take
\[
 g_0=\frac1{1000000},\qquad
 \Lambda_0=2916000000.
\]
Then
\begin{align*}
 \mathcal A_0(y)={}&1209956550+2343372300y\\
                  &+1740394800y^2+794068800y^3,\\
 \mathcal B_0(y)={}&-1135891798-1173235148y+258347816y^2,\\
 \mathcal C_0(y)={}&266598135+244982070y.
\end{align*}
The coefficients of
\(4\mathcal A_0\mathcal C_0-\mathcal B_0^2
=\sum_{r=0}^4 d_{0,r}y^r\) are
\begin{align*}
 d_{0,0}&=38461880464196,\\
 d_{0,1}&=1019289016601843792,\\
 d_{0,2}&=3362710436662080432,\\
 d_{0,3}&=3158456603732169536,\\
 d_{0,4}&=711386879353694144.
\end{align*}

\subsection{The final stage}

At \eqref{eq:stage-one}, take
\[
 g_1=\frac1{10000000},\qquad
 \Lambda_1=1457125331220000000.
\]
Then
\begin{align*}
 \mathcal A_1(y)={}&604514027207496297
                     +1171060420363058502y\\
                  &+869869372328057808y^2
                     +396910005991680000y^3,\\
 \mathcal B_1(y)={}&-567451842451072679
                     -586296072701709724y\\
                  &+128972962522474840y^2,\\
 \mathcal C_1(y)={}&133165863887059575
                     +122393853878983830y.
\end{align*}
The coefficients of
\(4\mathcal A_1\mathcal C_1-\mathcal B_1^2
=\sum_{r=0}^4 d_{1,r}y^r\) are
\begin{align*}
 d_{1,0}&=937158609712709801654383338059,\\
 d_{1,1}&=254346722868854219618398387158475448,\\
 d_{1,2}&=839298823187405114798749944189373584,\\
 d_{1,3}&=788518797558558572150007715039666880,\\
 d_{1,4}&=177683356043985413856187219704574400.
\end{align*}

\subsection{Why the certificates prove the inequalities}

Every displayed coefficient of \(\mathcal A_j\)
and of \(4\mathcal A_j\mathcal C_j-\mathcal B_j^2\)
is positive.
For \(y\ge0\) and every real \(\delta\),
\[
 \begin{split}
 4\mathcal A_j
   (\mathcal A_j\delta^2+\mathcal B_j\delta+\mathcal C_j)
 ={}&(2\mathcal A_j\delta+\mathcal B_j)^2\\
    &+4\mathcal A_j\mathcal C_j-\mathcal B_j^2>0.
 \end{split}
\]
The multiplier on the left of
\eqref{eq:certificate-template} is positive
for \(0\le\delta\le5/6\), \(0\le x\le1/2\).
Therefore \(E_j<-g_j\) on the entire
rectangle. No finite sampling argument
is used.

\section{Geometry and the remaining row ranges}
\label{app:ranges}

The endpoint certificate controls intermediate
nonprincipal rows above the detector floor.
For completeness, we verify its scale extension
and prime supply, then treat the small-\(\delta\)
classes, the floor, small rows, distant rows, and
the principal remainder. These cases account for
all terms in the high representation.

Use \(d_{\min}=1/100\), and set
\[
 \begin{array}{c|cc}
 j&\zeta_j&\text{retained high margin}\\ \hline
 0&10^{-7}&1/2000000\\
 1&10^{-9}&1/20000000.
 \end{array}
\]
In both cases
\[
 g_j-2\zeta_j>
 \text{the displayed retained margin},\qquad
 \frac{\ell_j}{h_j+\zeta_j}>\frac{21}{100}.
\]
Thus the count and short-block gate apply
throughout \(d_{\min}\le d\le h_j+\zeta_j\),
with a mesh chosen for the entire range of
\(\ell_i/d\).
At both geometries,
\[
 6\kappa_j\ell_j>3/4,\quad
 \ell_j/h_j<1/(6\kappa_j),\quad
 l_{y,j}<1/2,\quad 2l_{y,j}>h_j.
\]
The low scale margins are
\[
 \begin{array}{c|cc}
 j&l_{x,j}-\ell_j&
            l_{y,j}-\ell_j-11b_j/6\\ \hline
 0&8243/40000&16079/120000\\
 1&2060527/10000000&669701/5000000.
 \end{array}
\]
All are positive, and \(1/6<\ell_j<1/5\).

\paragraph{Small \(\delta\).}
Proposition~\ref{prop:row-count} and \(x\le1/2\)
give from \eqref{eq:small-delta-endpoint}
\[
 E_{\rm small}\le C_0+\delta
 \left\{\frac12+\frac{\ell}{2}
          (1/2+2\ell/h)-\frac h2\right\}.
\]
The coefficient of \(\delta\) is positive.
At \(\delta=1/8\), the two upper bounds are
\[
 -\frac{191530333}{45730080000},\qquad
 -\frac{47785783297087}{11432841120000000},
\]
respectively. Both are less than \(-1/250\).
These are the endpoint scalar checks required by
Proposition~\ref{prop:high}, which already includes
the extension to \(h_j+\zeta_j\).

\paragraph{The floor.}
At \(\delta_0=1/50\), use the trivial count
\(R=1\), \(q\le1/100\), and the old
pointwise reflected numerator contour.
Its endpoint is
\[
 C_0+(1/2+3\ell/2)/50.
\]
At the two geometries it equals
\[
 -\frac{5837}{500000},\qquad
 -\frac{8748053}{750000000},
\]
both less than \(-1/100\).
Its frequency slope is positive and less
than \(2\), so the same extension is harmless.

\paragraph{Small and bounded nonprincipal rows.}
Move to
\[
 \Rea s=\beta_*+e,\qquad
 \Rea w=1/2,\qquad \Rea z=17/50.
\]
Because \(\beta_*>B_j>87/100\) in the
continuation contradiction, this path is
inside the first Euler region with
\(\epsilon_0=37/100\).
The full correction has the crude bound
\[
 |\mathfrak H_{\eta,u,Z}|
       \ll(\N u)^{1/2+\eps}Z^{\ell z_0}
\]
on its retained strip: selected good-prime
factors are bounded, and the product of
selected ramified factors costs at most
\((\N u)^{1/2}\).
The source Hecke bound at \(w=1/2\) costs
at most \((\N u)^{3/5}\), with a fixed
polynomial height factor.
Retaining \((\N u)^{-z_0}\) and summing
the rows up to \(Z^{d_{\min}}\) is bounded
by \(Z^{2d_{\min}+\eps}\).
Consequently the exponent relative to
\(C_b(\beta_*)\) is at most
\[
 h(z_0-1/6)-l_y/2+2d_{\min}.
\]
Its values are
\[
 -\frac{433393}{6000000},\qquad
 -\frac{108354377}{1500000000},
\]
both less than \(-1/20\).
This includes the bounded nonprincipal
unit rows omitted from the moment argument.

\paragraph{Distant rows.}
For \(\N u>Z^{h+\zeta}\), use the absolute
lines \(s=w=2\) and \(z=z_\infty>2\).
The unchanged full-correction estimate
gives a tail exponent
\[
 B_{\rm abs}+(h+\zeta)(1+\eps)
                         -\zeta z_\infty+\eps,
 \qquad B_{\rm abs}=l_x/2+1+l_y.
\]
A sufficiently large fixed \(z_\infty\)
gives any required saving. It is chosen
before the target and before the external
tail order.

\paragraph{The principal row.}
The margins in Lemma~\ref{lem:principal-remainder}
are \(l_y/20\), \(h/600\), and any \(\mu\)
in \eqref{eq:principal-mu}.
They are strictly positive at both
geometries. Their exact first two values
are
\[
 \begin{array}{c|cc}
 j&l_{y,j}/20&h_j/600\\ \hline
 0&55163/2400000&31757/24000000\\
 1&13791227/600000000&2646491/2000000000.
 \end{array}
\]

\section{Common choices for the high-side estimate}
\label{app:choices}

The separate estimates must hold for one common
choice of probe. In particular, the prime windows
must be fixed before taking \(Z\) large, and the
final saving must be independent of the target
character. This appendix details the compatible
order of choices used in Proposition~\ref{prop:high}.
The continuation argument itself is contained in
Lemma~\ref{lem:closure}.

Fix an admissible parameter tuple and suppose
\(\Delta=\beta_*-B(\ell)>0\), with
\(\beta_*\le(1+\kappa)/2\). Use the retained
margin \(g_*\) from its admissibility conditions.

First choose the common moment losses,
capacity decrements and amplitude-rounding
allowance inside fixed small fractions
of \(g_*\) and \(\Delta\).
The source moment then supplies a mesh
independent of the slot count.
Choose a fixed even \(K\) large enough
that every \(\ell_i/d_{\min}\) fits that
mesh, and take disjoint nonnegative
smooth windows of those lengths.
Equal lengths with disjoint fixed
subintervals of the normalized annulus
are one possible choice.

After \(K\) is fixed, choose
\[
 \mu=\frac{87}{200}\min_i\ell_i,\qquad
 \gamma=\frac1{100}
       \min\{\Delta,g_*,l_y/20,h/600,\mu\}>0.
\]
Choose the bin width \(e\) and all
remaining small real losses so that
they consume only fixed small fractions
of these margins.
These choices are independent of the target.

For each fixed target, retain exactly
the source arithmetic data and choose
all required internal seminorm and
height orders.
Only then choose \(\tau>0\), with
\(\tau\le d_{\min}/100\), so that
the fixed height powers in the central
estimates are absorbed within the
reserved margins when \(T_1=Z^\tau\).
Assign fixed shares of the cumulative
height allowance to the detector
frequencies. Choose the external
tail order \(N\) afterwards, large
enough that every
\(Z^{B_\eta}T_1^{-N}\) is negligible.
Increasing \(N\) changes external test
seminorms, not an already fixed
arithmetic height exponent.
The physical probe and its scalar
normalizer do not depend on \(T_1\).

The high bounds now give
\[
 |J_\eta(Z)-f_\eta(Z)|\ll_\eta
       Z^{C_b(\beta_*)-2\gamma}.
\]
All saving exponents are common to
the targets; only constants, height
choices and thresholds may depend on
their fixed arithmetic data.

\section{Index of analytic inputs}\label{app:sources}

This index collects the source references used in
Section~\ref{sec:inputs} and the later proofs.
All statement numbers for \cite{OAI26} refer to its
September~30, 2026 version. Besides its \(7/8\)
conclusion, the imported inputs are as follows:
\begin{enumerate}
\item The character conventions, Hecke growth,
      logarithmic control, Euler-factor bounds, smooth
      separation, and finite-height Sobolev estimates
      of \cite[Section~4]{OAI26}.
\item The completed probe and its local identities
      in \cite[Sections~6--7 and 12--16]{OAI26},
      the reflected-energy estimate
      \cite[Lemma~14.3]{OAI26}, and the low separation
      and additive Gram bound
      \cite[Section~15.1 and Proposition~15.2]{OAI26}.
\item The simultaneous witnesses
      \cite[Proposition~8.3]{OAI26}, the marked and
      amplified inverse moments
      \cite[Lemmas~17.1 and 17.6]{OAI26}, and the
      plain fourth moment \cite[Lemma~18.1]{OAI26}.
\item The pointwise witness and prime-amplitude bounds
      \cite[Lemmas~8.2 and 19.1]{OAI26}, together with
      the buffered-bin construction used in
      Section~\ref{sec:inputs}.
\item The contour and continuation arguments
      \cite[Lemmas~10.3--10.6]{OAI26} and the transfer
      to Dirichlet \(L\)-functions
      \cite[Proposition~11.3]{OAI26}.
\end{enumerate}
The high-side comparison for the present probe is
derived in Section~\ref{sec:high}; it is not an
application of the final parameter comparison in
\cite{OAI26} to the new lengths.

\end{document}